\documentclass[reqno]{amsart}
\usepackage{xcolor}
\usepackage{amsmath, amssymb}
\usepackage{latexsym}

\usepackage{wasysym}
\usepackage{graphicx}

\newtheorem{Theorem}{Theorem}[section]
\newtheorem{Lemma}[Theorem]{Lemma}
\newtheorem{Proposition}[Theorem]{Proposition}
\newtheorem{Corollary}[Theorem]{Corollary}
\newtheorem{Remark}[Theorem]{Remark}

\newtheorem{Definition}[Theorem]{Definition}
\numberwithin{equation}{section}

\newcommand{\la}{\langle}
\newcommand{\ra}{\rangle}

\newcommand{\beq}{\begin{equation}}
\newcommand{\eeq}{\end{equation}}
\newcommand{\bes}{\begin{equation*}}
\newcommand{\ees}{\end{equation*}}

\def\N{{\Bbb N}}
\def\Z{{\Bbb Z}}
\def\R{{\Bbb R}}
\def\T{{\Bbb T}}
\def\C{{\Bbb C}}

\begin{document}

\title[Reducibility of quantum harmonic oscillator on $\R^d$]{Reducibility of quantum harmonic oscillator on $\R^d$ with  differential and quasi-periodic in time potential}
\author{ Zhenguo Liang and Zhiguo Wang}

\address {School of Mathematical Sciences and
Key Lab of Mathematics for Nonlinear Science, Fudan University,
Shanghai 200433, China} \email{zgliang@fudan.edu.cn}

\address {School of Mathematical Sciences,
Soochow University, Suzhou
215006,  China} \email{zgwang@suda.edu.cn}

\thanks{The first author was partially supported by NSFC grants 11371097, 11571249; the second author was partially supported by NSFC grants  11571249, 11671192.}


\date{}

\begin{abstract} We improve the results by Gr\'ebert and Paturel in \cite{GP} and  prove that a linear   Schr\"odinger equation on $\R^d$  with harmonic potential $|x|^2$ and small $t$-quasiperiodic potential as
$$
{\rm i}u_t - \Delta u+|x|^2u+\varepsilon V(\omega t,x)u=0, \ (t,x)\in \R\times\R^d
$$
reduces to an autonomous system for most values of the frequency vector $\omega\in\R^n$. The new point is that  the potential $V(\theta,\cdot )$ is  only in ${\mathcal{C}^{\beta}}(\T^n, \mathcal{H}^{s}(\R^d))$ with $\beta$ large enough. As a consequence any solution of such a linear PDE is almost periodic in time and remains bounded in some suitable Sobolev norms.\\

\indent\quad\quad\quad\quad  reducibility\\

\end{abstract}

\today

\maketitle

\section{Introduction}\label{introduction}

\subsection{State of  Reducibility Problem and Main Results}\label{s1.1}

We consider the following nonautonomous
linear  equation in $\R^d$
\begin{eqnarray}\label{HOeq}
{\rm i}u_t - \Delta u+|x|^2u+\varepsilon V(\omega t,x)u=0, \  u=u(t,x),\ (t,x)\in \R\times\R^d.
\end{eqnarray}
Here $\varepsilon>0$ is a small parameter and the frequency vector $\omega$ of  forced oscillator is regarded as a parameter in $D_0=[0,2\pi]^n\subset \R^n$.
The function $V$ is a real multiplicative potential which  is quasiperiodic in time. Namely, $V$ is a  continuous function of $(\theta,x)\in\T^n\times\R^d$.
We assume $V(z,\cdot)\in {\mathcal{C}^{\beta}}(\R^n, \mathcal{H}^{s}(\R^d))$ which will be denoted in the following(see Definition \ref{Cbetaspace}).\\
\indent As the usual reducibility results we consider the previous equation as a linear non-autonomous equation in the complex Hilbert space $L^2(\R^d)$ and we prove that it reduces to an autonomous system for most values of the frequency vector $\omega$.\\
\indent Similar as Gr\'ebert and Paturel \cite{GP}, we introduce some notations. Let $T:=-\Delta+|x|^2=-\Delta+x_1^2+\cdots+x_d^2$ be the d-dimensional quantum harmonic oscillator. Its spectrum is the sum of $d$ copies of   odd integers, i.e., the spectrum of $T$ equals to
$\widehat{\mathcal{E}}:=\{d,d+2,d+4,\cdots\}$.
For $j\in\widehat{\mathcal{E}}$, we denote the associated eigenspace by $ {E}_j$ whose dimension is
$$d_j:=card\{(i_1,\cdots,i_d)\in(2\N-1)^d\ |\ i_1+\cdots+i_d=j\}\leq j^{d-1}.$$
We denote $\{\Phi_{j,l},\ l=1,2,\cdots,d_j\}$, the basis of $ {E}_j$ obtained by $d-$tensor product of Hermite functions:
$\Phi_{j,l}=\varphi_{i_1}\otimes\cdots\otimes\varphi_{i_d}$
for some choice of ${i_1}+\cdots+{i_d}=j.$ Then setting
$$\mathcal{E}:=\{(j,l)\in\widehat{\mathcal{E}}\times\N\ |\ l=1,\cdots,d_j\}.$$
$(\Phi_{a})_{a\in\mathcal{E}} $ is a basis of $L^2(\R^d)$ and denote
$w_{j,l}:=j\ {\rm for}\ (j,l)\in\mathcal{E}$.
We have
\begin{equation}\label{eigenfunction1}
T\Phi_{a}=w_a\Phi_{a},\ a\in\mathcal{E}.
\end{equation}
We define in $\mathcal{E}$ an equivalence relation
$a\sim b\Leftrightarrow w_a=w_b$
and denote by $[a]$ the equivalence class associated to $a\in\mathcal{E}.$ Note that  $card\ [a]\leq w_a^{d-1}.$\\
\indent Let $s\geq 0$ be an integer we define
$$ \mathcal{H}^s:=\{u\in \mathcal{H}^s(\R^d,\C)\ |\  x\mapsto x^{\alpha_1}\partial_x^{\alpha_2}u \in L^{2}(\R^d)\ {\rm for\ any}\ \alpha_1,\alpha_2\in\N^d,\     0\leq|\alpha_1|+|\alpha_2|\leq
s \}.$$
We need to point it out that $\mathcal{H}^{s}$ is the form domain of $T^s$ and the domain of $T^{\frac{s}{2}}$ and  this allows us to extend the definition of $\mathcal{H}^{s}$ to any nonnegative real values of $s$(see Delort \cite{Del14}). \\
\indent To a function $u\in\mathcal{H}^s$ we associate the sequence $\xi$ of its Hermite coefficients by the formula $u(x)=\sum_{a\in\mathcal{E}}\xi_a\Phi_{a}(x).$
Then we define
$\ell_{s}^2:=\{(\xi)_{a\in\mathcal{E}}\ |\  \sum_{a\in\mathcal{E}}w_a^s|\xi_a|^2 <\infty\}$,
and for $s\geq0$,
$u\in\mathcal{H}^s\Leftrightarrow\xi\in\ell_{s}^2$.
Then we endow both spaces with the norm
$\|u\|_s=\|\xi\|_s=(\sum_{a\in\mathcal{E}}w_a^s|\xi_a|^2)^{\frac12}$.
If $s$ is a nonnegative integer, we will use the fact that the norm on $\mathcal{H}^s$ are equivalently defined as $\|T^{\frac{s}{2}}f\|_{L^{2}(\R^d)}$ and $\sum\limits_{0\leq|\alpha_1|+|\alpha_2|\leq
s}\|x^{\alpha_1}\partial_x^{\alpha_2}f\|_{ L^{2}(\R^d)}. $\\
To introduce the main result we introduce some notations and definitions.
\begin{Definition}\label{Cbetaspace}
Assume that $X$ is a complex Banach space with the norm $\|\cdot \|_{X}$. Let $\mathcal{C}^{b}(\R^n,X)$, $0<b<1$, be the space of
H\"older continuous functions $f : \R^n\rightarrow X$ with the norm
$$\|f\|_{\mathcal{C}^{b}(\R^n, X)} : = \sup\limits_{0<|z_1-z_2|<2\pi}\frac{\|f(z_1)-f(z_2)\|_{X}}{{|z_1-z_2|^{b}}}+\sup\limits_{z\in \R^n} \|f(z)\|_{X}.$$
If $b=0$, then $\|f\|_{\mathcal{C}^{b}(\R^n, X)}$ denotes the sup-norm. For $\beta=[\beta]+b$ with $0\leq b<1$, we denote
by ${\mathcal{C}^{\beta}}(\R^n, X)$ the space of functions $f: \R^n\rightarrow X$ with H\"older continuous partial derivatives and $\partial^{\alpha} f\in \mathcal{C}^{b}(\R^n, X_{\alpha})$ for all
multi - indices $\alpha=(\alpha_1, \cdots, \alpha_n)\in \N^n$, where $|\alpha| : = |\alpha_1|+\cdots+|\alpha_n| \leq \beta$ and $X_{\alpha}=\mathfrak{L}(\prod\limits_{i=1}^{|\alpha|}Y_i, X)$ with the standard norm and $Y_i : =\R^n$, $i=1, \cdots, |\alpha|$.
We define the norm
$
\|f\|_{\mathcal{C}^{\beta}(\R^n, X)} := \sum\limits_{|\alpha|\leq \beta}\|\partial^{\alpha}f\|_{\mathcal{C}^{b}(\R^n, X_{\alpha})}.
$
If a function $f$ has a finite norm $\|f\|_{\mathcal{C}^{\beta}(\R^n, X)}$, then we call $f\in {\mathcal{C}^{\beta}}(\R^n, X)$. \\
\indent Denote by $\mathcal{C}^{\beta}(\T^n, X)$ the space of all functions $f\in \mathcal{C}^{\beta}(\R^n, X)$ that are of period $2\pi$ in all variables. We define
$\|f\|_{\mathcal{C}^{\beta}(\T^n, X)} :=  \|f\|_{\mathcal{C}^{\beta}(\R^n, X)} $.
\end{Definition}

{\begin{Definition}\label{def1.1}
A real potential $V:\T^n\times\R^d\ni(\theta,x)\mapsto V(\theta,x)$ is called $(s,\beta)-$admissible if $V(\theta, x)\in \mathcal{C}^{\beta}(\T^n, \mathcal{H}^s(\R^d))$
with a finite norm, namely, $\|V(\theta, \cdot )\|_{\mathcal{C}^{\beta}(\T^n, \mathcal{H}^{s}(\R^d))}\leq C$,
where
\begin{equation*}\left\{
\begin{array}{cc}
s\geq0,& d=1,\\
s>2(d-2), & d\geq2,
\end{array}\right.
\end{equation*}
and the constant $C$ depends on $s, \beta, n$ and $d$.
\end{Definition}}
Set   $\gamma_1= n+d+2,\
\gamma_2=\frac{\alpha }{4+d+2\alpha }$(depending only on $s$ and $d$) and $\alpha$ given by (\ref{alpha}), we have
{\begin{Theorem}\label{quantumth}
Assume that the potential $V: \T^n\times \R^d\ni (\theta,x)\mapsto\R$ is $(s,\beta)-$admissible.  There exists $\delta$ satisfying  $0<\delta<\frac{\gamma_2}{24}$ and $\varepsilon_*(\beta,n,s,d,\delta)>0$, if   $0<\varepsilon<\varepsilon_*$,
$\beta>\max\{9(2+\frac{d}{\alpha})\frac{\gamma_1}{\gamma_2-24\delta},\ 9n,\ 12(d+1)\}$,  then there exists $D_\varepsilon\subset D_0$ with ${\rm Meas}(D_0\setminus D_\varepsilon)\leq c(\beta,n,d,s,\delta)\varepsilon^{\frac{3\delta}{2+\frac{d}{\alpha}}}$, such that for
 all $\omega\in D_\varepsilon$ the linear Schr\"odinger equation  (\ref{HOeq})
reduces to a linear autonomous equation in the space $\mathcal{H}^{s'}$ with $1\leq s'\leq \max\{s,1\}$.\\
\indent More precisely, for $\omega\in D_\varepsilon$, there exist a linear isomorphism $\Psi_\omega^\infty(\theta)\in\mathfrak{L}(\mathcal{H}^{s'})$ for $0\leq s'\leq s$, unitary on $L^2(R^d)$, where  $\Psi_\omega^\infty(\theta)\in \mathcal{C}^\mu(\T^n,\mathfrak{L}(\mathcal{H}^{s'}))$ for $0\leq s'\leq s$ with $\mu\notin\Z$ and $\mu\leq\frac{2}{9}\beta,$ and a bounded Hermitian operator $W=W_{\omega,\varepsilon}\in\mathfrak{L}(\mathcal{H}^{s'})$ such that $t\mapsto u(t,\cdot)\in\mathcal{H}^{s'}$ with $1\leq s'\leq \max\{s,1\}$  satisfies (\ref{HOeq}) if and only if $t\mapsto v(t,\cdot)= \Psi_\omega^\infty(\omega t)u(t,\cdot) $ satisfies the autonomous equation
\begin{equation*}\label{reducedeq}
\mathrm{i}\partial_t v -\Delta v+|x|^2 v+\varepsilon W(v)=0.
\end{equation*}
Furthermore, for $0\leq s'\leq s$,
$$\|\Psi_\omega^{\infty}(\theta)-id\|_{\mathcal{C}^{\mu}(\T^n, \mathfrak{L}(\mathcal{H}^{s'},\mathcal{H}^{s'+2\alpha}))}\leq C \varepsilon^{\frac{3}{2\beta}(\frac{2}{9}\beta-\mu)},\  (\theta,\omega)\in\T^n\times D_\varepsilon.$$
On the other hand, the infinite matrix
$(W_{a}^b)_{a,b\in\mathcal{E}}$ of the operator $W$ written in the Hermite basis($W_{a}^b=\int_{\R^d} \Phi_a W(\Phi_b)dx$) is block diagonal, i.e.
$W_{a}^b=0\ {\rm if}\  w_a\neq w_b.$
Denote $(\la V \ra_{a}^b)_{a,b\in\mathcal{E}}$ be the corresponding  infinite matrix of the operator $\la V \ra(x)=\frac{1}{(2\pi)^n}\int_{\T^n}V(\theta,x)d\theta$ written in the Hermite basis, we have
$$\|(W_{a}^b)_{a,b\in\mathcal{E}}-\Pi((\la V \ra_{a}^b)_{a,b\in\mathcal{E}})\|_{\mathfrak{L}(\ell_{s'}^2)}\leq  c\varepsilon^{\frac12}$$
 for $0\leq s'\leq s$, where $\Pi$ is the projection on the diagonal blocks.
\end{Theorem}}

\begin{Remark} Comparing with Theorem 1.2  in \cite{GP} we prove the reducibility theorem in the space $\mathcal{H}^{s'}$ with $1\leq s'\leq \max\{s,1\}$,  not in the energy space $ \mathcal{H}^{1}$. See Lemma \ref{convergence04} for details.
\end{Remark}

As a consequence of Theorem \ref{quantumth}, we prove the following corollary concerning the solutions of (\ref{HOeq}).
\begin{Corollary}\label{coro01}
Assume all the assumptions in Theorem \ref{quantumth} hold. Let $1\leq s'\leq \max\{s,1\}$ and let  $u_0\in\mathcal{H}^{s'}$, then there exists $\varepsilon_*>0$ such that for $0<\varepsilon<\varepsilon_*$ and $\omega\in D_\varepsilon$(in Theorem \ref{quantumth}), there exists a unique solution $u\in \mathcal{C}(\R,\mathcal{H}^{s'})$ of (\ref{HOeq}) such that $u(0)=u_0$. Moreover, u is almost periodic in time and satisfies\\
\begin{equation*}\label{utnorm}
(1-c\varepsilon) \|u_0\|_{\mathcal{H}^{s'}}\leq \|u(t)\|_{\mathcal{H}^{s'}}\leq(1+c\varepsilon )\|u_0\|_{\mathcal{H}^{s'}},\ \forall \ t\in\R
\end{equation*}
with some $c=c(s',s,d).$
\end{Corollary}

Consider on $L^2(\T^n)\otimes L^2(\R^d)$ the Floquet Hamiltonian operator
\begin{equation*}\label{floq}
K:=-\mathrm{i}\sum_{k=1}^n \omega_k\frac{\partial}{\partial\theta_k}-\Delta+|x|^2+\varepsilon V(\theta,x),
\end{equation*}
we have
\begin{Corollary}\label{coro02}
Assume all the assumptions in Theorem \ref{quantumth} hold.  Then   there exists $\varepsilon_*>0$ such that for $0<\varepsilon<\varepsilon_*$ and $\omega\in D_\varepsilon$, the spectrum of the Floquet operator $K$ is pure point.
\end{Corollary}

\subsection{Related results.}
The equations (\ref{HOeq})  can be generalized into a time-dependent Schr\"odinger equation
\begin{eqnarray}\label{NLStime3}
 {\rm i} \partial_{t}\zeta(t)=(\mathcal{A}+\varepsilon \mathcal{B}(\omega t))\zeta(t),
 \end{eqnarray}
where $\mathcal{A}$ is a positive self-adjoint operator on a separable Hilbert space $\mathcal{H}$ and the perturbation $\mathcal{B}$ is an operator-valued function  from $\T^n$ into the space of
symmetric operators on $\mathcal{H}$. Our aim is to show that for sufficiently small $\varepsilon$, and for $\omega$ belonging to a set of large measure, there exists a unitary transformation which conjugates Eq. (\ref{NLStime3}) to a time independent equation. If  this is true, we will call Eq.  (\ref{NLStime3}) is \textsl{reducible}. From the reducibility of Eq. (\ref{NLStime3}) and relative properties of the transformation we can easily prove the boundedness of the Sobolev norms and pure point spectrum of the relative Floquet operator, which is
defined by
\begin{eqnarray*}
K_{F} : = -{\rm i}\omega \cdot  \partial_{\theta}+ \mathcal{A} +\varepsilon \mathcal{B}(\theta)\quad {\rm on}\ \mathcal{H} \otimes L^2(\T^n).
\end{eqnarray*}
It has been proved in \cite{BEL, BLE, DS, DSV, GY00, H, JOY, N} that the Floquet operator $K_{F}$ is of pure point spectra or no absolutely continuous spectra where $\mathcal{B}$ is bounded. When $\mathcal B$ is unbounded,  the first result was obtained by Bambusi and  Graffi \cite{BG} where they considered  the time dependent Schr\"odinger equation
\begin{eqnarray*}\label{NLStime1}
{\rm i}\partial_{t} \psi(x,t) = H(t) \psi(x,t), x\in \R;  \qquad H(t) : = -\frac{d^2}{dx^2} + Q(x) +\varepsilon V(x,\omega t),\ \varepsilon\in \R,
\end{eqnarray*}
where $Q(x) \sim |x|^{2\alpha}$ with $\alpha>1$ as $|x|\rightarrow \infty$ and $|V(x,\theta)||x|^{-\beta}$ is bounded as $|x|\rightarrow \infty$ for some $\beta<\alpha-1$.
This entails the pure-point nature of the  spectrum of  the Floquet operator
\begin{eqnarray*}\label{floquetspectrum1}
K_{F} : = -{\rm i}\omega \cdot \partial_{\theta}-\frac{d^2}{dx^2}+Q(x)+\varepsilon V(x,\theta),
\end{eqnarray*}
on $L^2(\R) \otimes L^2(\T^n)$ for $\varepsilon$ small. Liu and Yuan \cite{LiuYuan0} solved the case when $\beta\leq \alpha-1$.  Very recently Bambusi \cite{Bam1,Bam2} solved the case when $\beta<\alpha+1$
under some additional assumptions. \\
\indent For 1-d quantum harmonic oscillator the main difficulty encountered by the traditional KAM method seems to be the eigenvalue spacing for the unperturbed operator does not grow. In \cite{EV}  Enss and Veselic proved that, if $\omega$ is rational,  the Floquet operator relative with the 1-d quantum harmonic oscillator has pure point spectrum when the perturbing potential $V$ is bounded and has sufficiently fast decay at infinity.
In \cite{Com87} Combescure obtained the reducibility under time periodic, spatially localized perturbation.
In \cite{Wang} Wang proved the spectrum of the Floquet operator $K$ is pure point for the quasiperiodic case where the perturbing potential
has \textit{exponential decay}. Greb\'ert and Thomann \cite{GT} improved the results in \cite{Wang} from exponential decay to \textit{polynomial decay}.
 In \cite{WLiang} we extended the results in \cite{GT} from polynomial decay to \textit{logarithmic decay}. Quite recently, in \cite{Bam1,Bam2} Bambusi dealt with the unbounded perturbation case for 1d harmonic oscillators. For example he can deal with the case $-\partial_{xx}+x^2+\varepsilon x a_1(\omega t)- {\rm i} a_2(\omega t)\varepsilon \partial_{x}$.  As Bambusi \cite{Bam1} pointed it out that his results didn't contradict with the interesting counterexamples in \cite{Del14} and  \cite{GY00}. \\
\indent  The results about the reducibility for higher spatial dimension are very few. In \cite{EK0} Eliasson and Kuksin obtained the reducibility for the Schr\"odinger equation on $\T^d$. In \cite{GP} Gr\'{e}bert and Paturel firstly  obtained the reducibility for any dimensional harmonic oscillator on $\R^d$ under the temporal quasiperiodic and analytic perturbation.  In this paper we will generalize the results in \cite{GP} from temporal
analytic perturbations  to differential perturbations.   \\
\indent Very recently, Bambusi, Greb\'ert, Maspero and Robert \cite{BaGrMaRo} proved a reducibility result for a quantum harmonic oscillator in  any dimension perturbed by a linear operator which is a polynomial of degree two in $x_j, -{\rm i\partial_j}$ with coefficients being real analytic in $\theta\in \T^n$. The proof depends on the following key fact: for polynomial Hamiltonians of degree at most 2 the correspondence between classical and quantum mechanics is exact(see also \cite{HLS}). But the reducibility problem keeps open for the quantum oscillator in arbitrary dimension with more general unbounded perturbations(see \cite{BaGrMaRo}, page 2). \\

\subsection{Brief description of the setting and main ideas of the proof.}
We use the notations introduced in \cite{GP}. In phase space $(u,\bar{u})\in\mathcal{H}^0\times\mathcal{H}^0$ endowed with the symplectic structure
$\mathrm{i}du\wedge d\bar{u}$, equation (\ref{HOeq}) is Hamiltonian with
\begin{eqnarray}\label{HOfun}
H=h(u,\bar{u})+\varepsilon p(\omega t,u,\bar{u}),
\end{eqnarray}
where
$$h(u,\bar{u})=\int_{\R^d}(u_x\bar{u}_x+|x|^2u\bar{u})dx,\ \ \ \ \ \   p(\omega t,u,\bar{u})=\int_{\R^d} V(\omega t,x)u\bar{u}dx.$$
 Expanding $u$ and $\bar{u}$ on the Hermite basis, $u(x)=\sum_{a\in\mathcal{E}}\xi_a\Phi_a(x),\ \bar{u}(x)=\sum_{a\in\mathcal{E}}\eta_a\Phi_a(x)$, the phase
space $(u,\bar{u})\in\mathcal{H}^0\times\mathcal{H}^0$ becomes into the phase space $(\xi,\eta)\in Y_0$(for the definition of $Y_{s}$ see  Subsection \ref{phase}). We endow $Y_0$ with the symplectic structure
$\mathrm{i}d\xi\wedge d\eta$. In this setting, (\ref{HOfun}) reads as
\begin{equation}\label{hameq000}
H(t,\xi,\eta)= \sum\limits_{ a\in\mathcal{E}}   w_a\xi_a\eta_a+ \varepsilon {p}_\omega(t,\xi,\eta)
\end{equation}
where
$
 {p}_\omega(t,\xi,\eta)=\la \xi,P(\omega t)\eta\ra=\sum_{a,b\in\mathcal{E}}P_a^b(\omega t)\xi_a\eta_b,
$
which is quadratic in $(\xi,\eta)$ with
\begin{eqnarray}\label{Pijform}
P_a^b(\omega t)=\int_{\R^d} V(\omega t,x)\Phi_a(x)\Phi_b(x)dx,\ \ a,b\in\mathcal{E}.
\end{eqnarray}
Therefore, the reducibility problem of system (\ref{HOeq}) is equivalent to the reducibility problem for the Hamiltonian system
\begin{eqnarray}
\left\{\begin{array}{c}
\dot{\xi}_a=-\mathrm{i}w_a\xi_a-\mathrm{i}\varepsilon (P^T(\omega t)\xi)_a,\\
\dot{\eta}_a
=\ \ \mathrm{i}w_a\eta_a+\mathrm{i}\varepsilon (P(\omega t)\eta)_a,
\end{array}\right.\ a\in\mathcal{E}\label{hs00}
\end{eqnarray}
associated to the non autonomous quadratic  Hamiltonian function (\ref{hameq000}). We will give a  general  reducibility result  in Subsection \ref{s2.4}
which can be applied to system (\ref{hs00}) and the proof is based on KAM theory. We remark that KAM theory is almost well-developed for nonlinear Hamiltonian PDEs in 1-d context. See \cite{BBP2, GY1, KLiang, KP, Ku0, Ku1, Ku2, LZ, LY1, LiuYuan, P2, ZGY} for 1-d KAM results.  Comparing with 1-d case, the KAM results for multidimensional PDEs are relatively few. Refer to  \cite{EGK, EK, GP1, GXY, GY2, PX} for n-d results. See \cite{Berti} for an almost complete picture of recent KAM theory.

\noindent\emph{Highlights.} By introducing $\theta=\omega t$, system (\ref{hs00}) is equivalent to an autonomous system with Hamiltonian
\begin{eqnarray}\label{hamit}
\mathcal{H}(\theta,y,\xi,\eta)
=\sum\limits_{j=1}^n\omega_jy_j+\la\xi,N_0\eta\ra + \varepsilon \la\xi, P(\theta)\eta\ra,\ \ \ \ \la\xi,N_0\eta\ra=\sum\limits_{a\in\mathcal{E}}w_a\xi_a\eta_a.
\end{eqnarray}
In  \cite{GP} Gr\'{e}bert and Paturel  assumed that the potential $V(\theta, \cdot)$ is real analytic with value in $\mathcal{H}^{s}(\R^d)$ with $s>2(d-2)\geq 0$ when $d\geq 2$. Here we only discuss the higher dimensional case for simplicity.
Then in Lemma 3.2  Gr\'{e}bert and Paturel  \cite{GP} proved that  $P( \theta)\in \mathcal{M}_{s,\alpha}(D_0, \sigma)$, where $\alpha>0$ is critical(see the definition of $\mathcal{M}_{s,\alpha}(D_0,\sigma)$ in Section \ref{section2}). In this paper we only assume that the potential $V(\theta, \cdot )\in \mathcal{C}^{\beta}(\T^n, \mathcal{H}^s(\R^d))$
with the same condition on $s$.  Using the techniques from \cite{GP} and functional analysis(\cite{Ber}) we can prove that
$P(\theta)\in {\mathcal{C}^{\beta}}(\T^n, \mathcal{M}_{s,\alpha})$(see Lemma \ref{L3.3}), which can be considered as a parallel lemma as Lemma 3.2 in \cite{GP}. \\
\indent Now our main problem is to build a similar reducibility result for the Hamiltonian (\ref{hamit}) when $P(\theta)\in C^{\beta}(\T^n, \mathcal{M}_{s,\alpha})$ and thus a smooth KAM is needed here.
We recall that the smoothing techniques were firstly introduced by Moser \cite{Moser1, Moser2} and developed later by many people, see
Salamon and Zehnder \cite{SaZe},  P\"oschel \cite{Pos2},  Chierchia and Qian \cite{CQ}, Berti and Bolle \cite{BB14, BB13} and etc.  An earlier reducibility result about time dependent Schr\"odinger operator with finite differentiable unbounded perturbation has been obtained by Yuan and Zhang in \cite{YZ13}.  In \cite{Bam1} Bambusi's method in dealing with the differential perturbations  is more close to the classical proof in \cite{Sal04}. A significant difference between our paper and \cite{YZ13}, \cite{Bam1} is that we deal with the matrix block, not the single matrix element. For the following proof we almost follow the presentation of \cite{CQ} in the spirit of \cite{Sal04} combined with the KAM method in  \cite{GP}. \\
\indent More clearly, we will introduce a series of analytic functions $P^{(\nu)}(\theta)\in \mathcal{M}_{s,\alpha}(D_0, \sigma_{\nu}),\ \nu=0,1,2,\cdots,$  and $P^{(\nu)}(\theta)\rightarrow   P(\theta)$ in $\theta\in \T^n$ as $\sigma_\nu$
shrinking to 0($\nu\rightarrow\infty$), see Lemmas \ref{PP} and \ref{l4.1} for details.\\
\indent  Thus, instead of considering  the original function $\mathcal{H}$, in each KAM step, we consider the analytic Hamiltonian function
\begin{eqnarray*}
H^{(\nu)}(\theta,y,\xi,\eta)
=\sum\limits_{j=1}^n\omega_jy_j+\la\xi,N_0\eta\ra+ \varepsilon\la\xi, P^{(\nu)}(\theta)\eta\ra
 \end{eqnarray*}
 which is an approximation  of (\ref{hamit}).
 We suppose that there exists symplectic map $\Phi^\nu$ such that
 \begin{eqnarray*}\label{everyterm}
H^{(\nu)}\circ \Phi^{\nu}=\sum\limits_{j=1}^n\omega_jy_j+\la\xi,N_\nu\eta\ra+ \la\xi, P_{\nu}(\theta)\eta\ra
\end{eqnarray*}
 with the norm of $P_{\nu}(\theta)$ is less than $\epsilon_{\nu}/2 $.
Then in $(\nu+1)^{th}$ step, we consider the Hamiltonian
\begin{eqnarray*}
H^{(\nu+1)}(\theta,y,\xi,\eta)
&=&\sum\limits_{j=1}^n\omega_jy_j+\la\xi,N_0\eta\ra+  \varepsilon\la\xi, P^{(\nu+1)}(\theta)\eta\ra\\
&=&H^{(\nu)} +(H^{(\nu+1)}-H^{(\nu)})
 \end{eqnarray*}
 which is tiny different from $H^{(\nu)}$.
 By $\Phi^\nu$  we have
\begin{eqnarray*}
H^{(\nu+1)}\circ \Phi^{\nu}=H^{(\nu)}\circ \Phi^{\nu}+(H^{(\nu+1)}-H^{(\nu)})\circ \Phi^{\nu}.
\end{eqnarray*}
We  shrink the radius of the analytic domain from $\sigma_\nu$ to $\sigma_{\nu+1}=\sigma_\nu^{\frac{3}{2}}$ in order to prove  that  the norm of additional quadratic perturbation
term $(H^{(\nu+1)}-H^{(\nu)})\circ \Phi^{\nu}$ is less than $\epsilon_{\nu}/2$ too(see Lemma \ref{L4.5}).
Then from Proposition 4.1(\cite{GP}) we can construct $\Phi_{\nu+1}$ such that
\begin{eqnarray*}
H^{(\nu+1)}\circ \Phi^{\nu+1}=H^{(\nu+1)}\circ \Phi^{\nu}\circ \Phi_{\nu+1}= h_{\nu+1}+\la
\xi,P_{\nu+1}(\theta)\eta\ra,
\end{eqnarray*}
where  $h_{\nu+1}$ is in normal form and the norm of $P_{\nu+1}$ is less than $ \epsilon_{\nu+1}/2$.
Thus we can formulate the iteration lemma. Finally,  let $\nu\rightarrow \infty$, we obtain
$\mathcal{H}\circ \Phi_{\omega}^{\infty}=\sum\limits_{j=1}^n\omega_jy_j+\la\xi,N_{\infty}\eta\ra$,
where $(\xi,\eta)\in Y_{s'}$ with $1\leq s'\leq \max\{s,1\}$ and
 $N_\infty(\omega)=\lim_{\nu\rightarrow\infty}N_{\nu}(\omega)$  in normal form
and with the norm close to $N_0$.
For obtaining the above proof  we need to show that $H^{(\nu)}\rightarrow \mathcal{H}$ and $\Phi^{\nu}\rightarrow \Phi_{\omega}^{\infty}$.
Furthermore, from the estimation of $\Phi^{\nu}-\Phi^{\nu-1}$ and Lemma \ref{smoothinginverse} we deduce that
$\Phi_\omega^\infty-id \in \mathcal{C}^{\mu}(\T^n, \mathfrak{L}(Y_{s'}, Y_{s' +2\alpha}))$ for all $0\leq s'\leq s$, where
$\mu\leq \frac{2}{9}\beta$ and is not an integer.

\indent The paper is organized as follows: In Sect. 2 we state the abstract reducibility theorem: Theorem \ref{MainTheorem}. In Sect. 3 we prove Theorem \ref{quantumth}, Corollary \ref{coro01} and Corollary \ref{coro02},
which are direct results from Theorem \ref{MainTheorem}. In Sect. 4 we prove Theorem \ref{MainTheorem}. The section is split into a few subsections. Finally, the appendix contains some technical lemmas.

\section{Reducibility Theorem for Quantum Harmonic Oscillator in $\R^d$ with Quasiperiodic in Time Potential: Smooth Version.}\label{section2}

\subsection{Setting}\label{phase}
{\noindent\emph{Notations.}}
Denote $\C$, $\R$, $\Z$, $\N$    be the set of all complex numbers, real numbers, integers and nonnegative integers, respectively. $\T=\R/ 2\pi\Z$. $\la \cdot,\cdot\ra$ is the standard scalar product in $\ell^2$, while
$\la f \ra:=\frac{1}{(2\pi)^n}\int_{\T^n}f(\theta)d\theta$ be the mean value of $f$ on the torus $\T^n$.
 $|\cdot|$ will be general to denote a supremum norm with a notable exception: for a multi-index
$k=(k_1,\cdots, k_n)\in\Z^n$, denote $|k|=\sum_{i=1}^n |k_i|$.
In the whole paper we use $\nu$   to stand for the KAM iteration step.

\noindent\emph{Linear space.} Following the notations in Subsection \ref{s1.1}, for $s\geq0,$ we consider the complex weighted $\ell^2-$space
$$\ell_{s}^2:=\{\xi=(\xi_a\in\C)_{a\in\mathcal{E}}\big|\  \|\xi\|_s <\infty\}$$
with $\|\xi\|_s^2:=\sum\limits_{a\in\mathcal{E}} |\xi_a|^2w_a^{s}$ and $\ell_{0}^2$ is $\ell^2$.
Then we define
 $$Y_s:=\ell_{s}^2\times\ell_{s}^2=\{\zeta=(\zeta_a=(\xi_a,\eta_a)\in\C^2)_{a\in\mathcal{E}}\big|\ \|\zeta\|_s<\infty\}$$
with  $\|\zeta\|_s^2:=\sum\limits_{a\in\mathcal{E}} (|\xi_a|^2+|\eta_a|^2) w_a^{s}$.\\
\indent We provide the space $Y_s,\ s\geq0,$ with the symplectic structure $\mathrm{i}\sum_{a\in\mathcal{E}}d \xi_a\wedge d \eta_a$. To any smooth function
$f(\xi,\eta)$ defined on a domain of $Y_s$, it corresponds to the Hamiltonian system:
\begin{eqnarray}
\left\{\begin{array}{c}
\dot{\xi} =-\mathrm{i}\nabla_{\eta}f(\xi,\eta)\\
\dot{\eta} =\ \ \mathrm{i}\nabla_{\xi}f(\xi,\eta)
\end{array}\right.
\end{eqnarray}
where $\nabla f=(\nabla_{\xi}f,\nabla_{\eta}f)^T$ is the gradient with respect to the scalar product in $Y_0.$\\
\indent For any smooth functions $f(\xi,\eta),\ g(\xi,\eta)$,   the Poisson bracket of $f$, $g$ is given by
\begin{eqnarray}
\{f,g\}:= \mathrm{i}\sum_{a\in\mathcal{E}}\left(\frac{\partial f}{\partial\xi_a}\frac{\partial g}{\partial\eta_a}-\frac{\partial g}{\partial\xi_a}\frac{\partial
f}{\partial\eta_a}\right).
\end{eqnarray}
\indent We   also consider the extended phase space
$\mathcal P_s:=\T^n\times \R^{n}\times Y_s\ni (\theta, y, \xi,\eta)$.
For smooth functions $f(\theta, y,\xi,\eta),\ g(\theta, y,\xi,\eta)$,  the  Poisson bracket is given  by
\begin{eqnarray}
\{f,g\}:=\sum_{j=1}^n\left(\frac{\partial f}{\partial y_j}\frac{\partial g}{\partial\theta_j}-\frac{\partial g}{\partial y_j}\frac{\partial
f}{\partial\theta_j}\right) + \mathrm{i}\sum_{a\in\mathcal{E}}\left(\frac{\partial f}{\partial\xi_a}\frac{\partial g}{\partial\eta_a}-\frac{\partial g}{\partial\xi_a}\frac{\partial
f}{\partial\eta_a}\right).
\end{eqnarray}

\noindent\emph{Infinite matrices.} We denote by $\mathcal{M}_{s,\alpha}$ the set of infinite matrices $A:\mathcal{E}\times\mathcal{E}\rightarrow \C$ with the norm
\begin{eqnarray*}
|A|_{s,\alpha}:=\sup_{a,b\in\mathcal{E}}\left(w_aw_b\right)^{\alpha}\left\|A_{[a]}^{[b]}\right\|\left(\frac{\sqrt{\min(w_a,w_b)}+|w_a-w_b|}{\sqrt{\min(w_a,w_b)}}\right)^{s/2}<+\infty,
\end{eqnarray*}
where $A_{[a]}^{[b]}$ denotes the restriction of $A$ to the block $[a]\times[b]$ and $\|\cdot\|$ denotes the operator norm.
We also denote $\mathcal{M}_{s,\alpha}^+$ be the subspace of $\mathcal{M}_{s,\alpha}$ satisfying that an infinite matrix $A\in\mathcal{M}_{s,\alpha}^+$ if
\begin{eqnarray*}
|A|_{s,\alpha+}:=\sup_{a,b\in\mathcal{E}}(w_aw_b)^{\alpha}\left(1+|w_a-w_b|\right)\left\|A_{[a]}^{[b]}\right\|\left(\frac{\sqrt{\min(w_a,w_b)}+|w_a-w_b|}{\sqrt{\min(w_a,w_b)}}\right)^{s/2}<+\infty.
\end{eqnarray*}
\indent From the definition we have following simple facts.
\begin{Lemma}\label{completed}
$(\mathcal{M}_{s,\alpha}, |\cdot|_{s,\alpha})$ and $(\mathcal{M}_{s,\alpha}^+, |\cdot|_{s,\alpha+})$ are Banach spaces.
\end{Lemma}
\begin{Lemma}[\cite{GP},\ Lemma 2.1]\label{daishu01}Let $0<\alpha\leq1$ and $s\geq0$, there exists a constant $c(\alpha,s)>0$ such that\\
i) If $A\in \mathcal{M}_{s,\alpha}$ and $B\in \mathcal{M}_{s,\alpha}^+$, then $AB$ and $BA$  belong to $\mathcal{M}_{s,\alpha}$ and
\begin{eqnarray*}
|AB|_{s,\alpha},\ |BA|_{s,\alpha}\leq c(\alpha,s)|A|_{s,\alpha}|B|_{s,\alpha+}.
\end{eqnarray*}
ii) If $A,B\in \mathcal{M}_{s,\alpha}^+$, then $AB$    belongs to $\mathcal{M}_{s,\alpha}^+$ and
\begin{eqnarray*}
|AB|_{s,\alpha+} \leq c(\alpha,s)|A|_{s,\alpha+}|B|_{s,\alpha+}.
\end{eqnarray*}
iii) If $A\in \mathcal{M}_{s,\alpha}$, then for any $t\geq1,\ A\in {\mathfrak{L}}(\ell^2_t,\ell^2_{-t})$  and
\begin{eqnarray*}
\|A\xi\|_{-t}\leq c(\alpha,s)|A|_{s,\alpha}\|\xi\|_t.
\end{eqnarray*}
iv) If $A\in \mathcal{M}_{s,\alpha}^+$, then $ A\in \mathfrak{L}(\ell_{s'}^2,\ell^2_{s'+2\alpha})$  for all $0\leq s'\leq s$ and
\begin{eqnarray*}
\|A\xi\|_{ s'+2\alpha}\leq c(\alpha,s)|A|_{s,\alpha+}\|\xi\|_{s'};
\end{eqnarray*}
furthermore, $A\in\mathfrak{L}(\ell^2_{1},\ell^2_{1})$ and
\begin{eqnarray*}
\|A\xi\|_{1}\leq c(\alpha,s)|A|_{s,\alpha+}\|\xi\|_{1}.
\end{eqnarray*}
\end{Lemma}
\noindent{\emph{Normal form.}} We introduce the following definitions.
\begin{Definition}
A matrix $F:\ \mathcal{E}\times\mathcal{E}\rightarrow\C$ is  Hermitian, i.e., $F_{b}^a=\overline{F_{a}^b},\ a,b\in\mathcal{E}.$
\end{Definition}
\begin{Definition}
A matrix $N:\ \mathcal{E}\times\mathcal{E}\rightarrow\C$ is in normal form, and we denote $N\in\mathcal{NF},$ if\\
(i) $N$ is Hermitian,\\
(ii) $N$ is block diagonal, i.e., $N_{b}^a=0,$ for $w_a\neq w_b$.
\end{Definition}
\noindent{\emph{Quadratic form.}} To a matrix $Q=(Q_a^b)_{a,b\in\mathcal{E}}\in\mathfrak{L}(\ell_t^2,\ell_{-t}^2)$ we associate in a unique way a quadratic form $q(\xi,\eta)$ on $Y_t$ by the formula
$q(\xi,\eta):=\la\xi,Q\eta\ra=\sum_{a,b\in\mathcal{E}}Q_{a}^b\xi_a\eta_b$ and the Poisson bracket
$$\{q_1,q_2\}(\xi,\eta)=-\mathrm{i}\la\xi,[Q_1,Q_2]\eta\ra$$
where $[Q_1,Q_2]=Q_1Q_2-Q_2Q_1$ is the commutator of   two matrices $Q_1$ and $Q_2$. Moreover, if $Q\in \mathcal{M}_{s,\alpha}$ then $$\sup_{a,b\in\mathcal{E}}\left\|(\nabla_\xi\nabla_\eta q)_{[a]}^{[b]}\right\|\leq\frac{|Q|_{s,\alpha}}{(w_aw_b)^{\alpha}}\left(\frac{\sqrt{\min(w_a,w_b)}}{\sqrt{\min(w_a,w_b)}+|w_a-w_b|}\right)^{s/2}.$$

\noindent\emph{Parameter.} In  the paper $\omega$ will play the role of a parameter belonging to $D_0=[0,2\pi]^n$.
All the constructed functions will depend on $\omega$ with $\mathcal{C}^1$ regularity. When a function is only defined on a Cantor subset of $D_0$ the regularity
is understood in  Whitney sense.

\noindent\emph{A class of quadratic Hamiltonians.} Let $D\subset D_0,\ s\geq0,\ \alpha>0$ and $\sigma>0$. We denote by $\mathcal{M}_{s,\alpha}(D,\sigma)$ the set of
mappings  as
$\T^n_\sigma\times D\ni (\theta,\omega)\mapsto Q(\theta,\omega)\in \mathcal{M}_{s,\alpha}$
which is real analytic on $\theta\in \T^n_\sigma:=\left\{\theta\in\C^n\ \big|\ |\Im \theta|<\sigma\right\}$ and $\mathcal{C}^1$
continuous on  $\omega\in D$. This space is equipped with the norm
$$[Q]_{s,\alpha}^{D,\sigma}:=\sup_{\omega\in D,|\Im \theta|<\sigma,|k|=0,1}\left|\partial^k_\omega Q(\theta,\omega)\right|_{s,\alpha}.$$
In view of Lemma \ref{daishu01}, to a matrix $Q\in\mathcal{M}_{s,\alpha}(D,\sigma)$, we can associate the quadratic form on $Y_t$ with $t\geq 1$
$$q(\xi,\eta;\omega,\theta)=\la\xi,Q(\omega,\theta)\eta\ra$$
and we have
$$|q(\xi,\eta;\omega,\theta)|\leq c(\alpha,s)[Q]_{s,\alpha}^{D,\sigma}\left\|(\xi,\eta)\right\|_t^2$$ for $(\xi,\eta)\in Y_t,\ \omega\in D,\ \theta\in\T^n_\sigma.$ \\
\indent The subspace of $ \mathcal{M}_{s,\alpha}(D,\sigma)$ formed by  $F(\theta,\omega)$ such that $\partial^k_\omega F(\theta,\omega)\in \mathcal{M}^+_{s,\alpha},\ |k|=0,1,$ is denoted by $ \mathcal{M}_{s,\alpha}^+(D,\sigma)$ and
equipped with the norm
$$[F]_{s,\alpha+}^{D,\sigma}:=\sup_{\omega\in D,|\Im \theta|<\sigma,|k|=0,1}\left|\partial^k_\omega F(\theta,\omega)\right|_{s,\alpha+}.$$
The subspace of  $  \mathcal{M}_{s,\alpha}(D,\sigma)$ that are independent of $\theta$ will be denoted by $\mathcal{M}_{s,\alpha}(D)$ and for $N\in \mathcal{M}_{s,\alpha}(D),$
$$[N]_{s,{\alpha}}^{D}:=\sup_{\omega\in D,|k|=0,1}|\partial^k_\omega N(\omega)|_{s,{\alpha}}.$$
From Lemma \ref{daishu01},   the following results hold.
\begin{Lemma}\label{daishu}For $0<\alpha\leq1$ and $s\geq0$, there exists a constant $c(\alpha,s)>0$ such that\\
i) If $A\in \mathcal{M}_{s,\alpha}(D,\sigma)$ and $B\in \mathcal{M}_{s,\alpha}^+(D,\sigma)$, then $AB$ and $BA$  belong to $\mathcal{M}_{s,\alpha}(D,\sigma)$ and
\begin{eqnarray*}
[AB]_{s,\alpha}^{D,\sigma},\ [BA]_{s,\alpha}^{D,\sigma}\leq c(\alpha,s)[A]_{s,\alpha}^{D,\sigma}[B]_{s,\alpha+}^{D,\sigma}.
\end{eqnarray*}
ii) If $A,B\in \mathcal{M}_{s,\alpha}^+(D,\sigma)$, then $AB$   belongs to $\mathcal{M}_{s,\alpha}^+(D,\sigma)$ and
\begin{eqnarray*}
[AB]_{s,\alpha+}^{D,\sigma} \leq c(\alpha,s)[A]_{s,\alpha+}^{D,\sigma}[B]_{s,\alpha+}^{D,\sigma}.
\end{eqnarray*}
Moreover, if $A\in \mathcal{M}_{s,\alpha}(D,\sigma)$  and $B\in \mathcal{M}_{s,\alpha}^+(D,\sigma)$ then  $Ae^{B}$, $e^{B}A\in \mathcal{M}_{s,\alpha}(D,\sigma)$  and $e^{B}-Id\in \mathcal{M}_{s,\alpha}^+(D,\sigma)$ satisfying
\begin{eqnarray}\label{exponorm01}
 [Ae^{B}]_{s,\alpha},\
 [e^{B}A]_{s,\alpha}&\leq& [A]_{s,\alpha}^{D,\sigma}e^{c(\alpha,s)[B]_{s,\alpha+}^{D,\sigma}}.
\\ \label{exponorm02}
 [e^{B}-Id]_{s,\alpha+}^{D,\sigma}&\leq& [B]_{s,\alpha+}^{D,\sigma}e^{c(\alpha,s )[B]_{s,\alpha+}^{D,\sigma}}.
\end{eqnarray}
\end{Lemma}

\noindent\emph{Hamiltonian flow.}
When $F$  depends smoothly on $\theta$, $\T^n\ni\theta\mapsto F(\theta)\in \mathcal{M}_{s,\alpha}^+$ with $0<\alpha\leq1$  we associate to $f=\la\xi, F(\theta)\eta\ra $ the symplectic transformation, generated by the time 1 map of $X_f$, on the extended phase space $\mathcal{P}_s: $
$$(\theta,y,\xi,\eta)\mapsto(\theta,\tilde{y},e^{-\mathrm{i}F^T}\xi,e^{\mathrm{i}F}\eta),$$ where
$\tilde{y}=y+\la \xi, \nabla_{\theta}F(\theta) \eta\ra$. In the following we will never calculate $\tilde{y}$ explicitly since the non homogeneous
Hamiltonian system (\ref{hs00}) is equivalent to the system (\ref{autohs}) where the variable conjugated to $\theta$ is not concerned.
Thus, the above symplectic transformation is rewritten into a symplectic linear change, restricted in on $Y_0$, which is  given by
$(\xi,\eta)\mapsto(e^{-iF^T }\xi,e^{iF }\eta)$.  It is well defined and invertible in $ Y_{0} $ as a
consequence  of Lemmas \ref{daishu01} and \ref{daishu}.   Recall that
a sufficient and necessary condition for this map to preserve the symmetry $\eta=\bar{\xi}$ is $F^T(\theta)=\overline{F}(\theta)$ when $\theta\in\T^n$, i.e., $F$ is a Hermitian matrix.

\noindent\emph{$\mathcal{C}^1$ norm of operator in $\omega$.} Given   $(\theta,\omega)\in\T_{\sigma}^n\times D$,  $\Phi(\theta,\omega)\in\mathfrak{L}(Y_s,Y_{s'})$ being $\mathcal{C}^1$
operator with respect to  $\omega$  in Whitney sense, we define the $C^1$ norm of $\Phi(\theta,\omega)$ with respect to $\omega$ by
$$\|\Phi\|^*_{\mathfrak{L}(Y_s,Y_{s'})}=\sup_{(\theta,\omega)\in\T_{\sigma}^n\times D,\ |k|=0,1,\ \|\zeta\|_s\neq0}\frac{\|\partial_\omega^k\Phi (\theta,\omega)\zeta\|_{s'}}{\|\zeta\|_s}.$$
 \subsection{The reducibility theorem}\label{s2.4}
 \noindent In this subsection we state an abstract reducibility theorem for quadratic $t$-quasiperiodic   Hamiltonian of the form
\begin{equation}\label{hameq}
H(t,\xi,\eta)= \la\xi, N_0\eta\ra+ \varepsilon\la\xi, P(\omega t)\eta\ra, \quad (\xi,\eta)\in Y_0,
\end{equation}
 and the associated Hamiltonian system  is
\begin{eqnarray}
\left\{\begin{array}{c}
\dot{\xi}=-\mathrm{i}N_0\xi-\mathrm{i}\varepsilon P^T(\omega t)\xi,\\
\dot{\eta}=\ \ \mathrm{i}N_0\eta+\mathrm{i}\varepsilon P (\omega t)\eta,\
\end{array}\right.\label{hameq00}
\end{eqnarray}
where $N_0=diag\{\lambda_a,\ a\in\mathcal{E}\}$ 
satisfying the following assumptions:\\
 \textbf{Hypothesis A1 - Asymptotics.} There exist   positive constants $c_0,\ c_1,\ c_2$ such that $$c_1 w_a\geq\lambda_a\geq c_2w_a
 \ {\rm and}\ |\lambda_a-\lambda_b|\geq c_0|w_a-w_b|,\ a,b\in\mathcal{E}.$$
\textbf{Hypothesis A2 - Second Melnikov condition in measure estimates.}
 There exist   positive constants $\alpha_1,\alpha_2$ and $c_3$ such that the following holds: for each $0<\kappa<1/4$ and $K>0$ there exists a closed
 subset $ D':= D'(\kappa,K)\subset  D_0$ with
 ${\rm Meas}( D_0\setminus  D')\leq c_3K^{\alpha_1}\kappa^{\alpha_2}$ such that for all $\omega\in  D',$  $k\in \Z^n$ with $0<|k|\leq K$ and  $a,b\in\mathcal{E}$ we have
 $$|\la k,\omega\ra +\lambda_a-\lambda_b|\geq \kappa(1+|w_a-w_b|).$$
\indent Then  we have the following reducibility results.
\begin{Theorem}\label{MainTheorem}
Given a non autonomous Hamiltonian (\ref{hameq}) with $ d\geq1$, we assume  that $(\lambda_{a})_{a\in \mathcal{E}}$ satisfies Hypothesis  A1-A2 and  $P(\theta)\in \mathcal{C}^\beta(\T^n,\mathcal{M}_{s,\alpha})$ with $s\geq0,\ \alpha>0$ and $\beta>\max\{9(2+\frac{d}{\alpha})\frac{\gamma_1}{\gamma_2-24\delta},\ 9n,\ 12(d+1)\}$ where  $\gamma_1= \max\{\alpha_1, n+d+2\},\
\gamma_2=\frac{\alpha\alpha_2}{4+d+2\alpha\alpha_2}$, $\delta\in(0,\frac{\gamma_2}{24})$.\\
\indent  Then there exists $\varepsilon_*(n,\beta,s,d,\delta )>0$ such that
if $0<\varepsilon<\varepsilon_*(n,\beta,s,d,\delta )$, there exist\\
(i) a Cantor set $D_\varepsilon\subset D_0$ with ${\rm Meas}(D_0\setminus D_\varepsilon)\leq c(n,\beta,d,s,\delta)\varepsilon^{\frac{3\delta}{2+d/\alpha}}$;\\
(ii) a $\mathcal{C}^1$ family in $\omega\in D_\varepsilon$(in Whitney sense), linear, unitary and symplectic
coordinate transformation
$\Phi_\omega^\infty(\theta): Y_0\rightarrow Y_0,\ \theta\in\T^n,\ \omega\in D_\varepsilon,$ of the form
\begin{equation*}\label{transf}
(\xi_+,\eta_+)\mapsto(\xi,\eta)=\Phi_\omega^\infty(\theta)(\xi_+,\eta_+)=( \overline{M}_\omega(\theta)\xi_+,{M}_\omega(\theta)\eta_+),
\end{equation*}
where
 $\Phi_{\omega}^{\infty}(\theta)-id \in \mathcal{C}^{\mu}(\T^n, \mathcal{L}(Y_{s'}, Y_{s'+2\alpha}))$($0\leq s'\leq s, \mu\leq\frac{2}{9}\beta$, $\mu\notin\Z$)  and satisfies
\begin{eqnarray*}
\|\Phi_\omega^\infty-id\|_{\mathcal{C}^{\mu}(\T^n, \mathfrak{L}(Y_{s'}, Y_{s'+2\alpha})) }&\leq&   C(n,\beta, \mu, d, s) \epsilon^{\frac{3}{2\beta}(\frac{2}{9}\beta-\mu)};
\end{eqnarray*}
(iii)  a $\mathcal{C}^1$ family of  autonomous quadratic Hamiltonians in  normal forms
$$H_\infty(\xi_+,\eta_+)=\la\xi_+,N_\infty(\omega)\eta_+\ra,\ \omega\in D_\varepsilon,$$
where $N_\infty(\omega)\in \mathcal{NF}$, in particular block diagonal (i.e. $N_{a}^b=0$ for $w_a\neq w_b$), and is close to
$N_0$, i.e. $[N_\infty(\omega)-N_0]_{s,\alpha}^{D_\varepsilon}\leq  c(n,\beta,d,s)\varepsilon, $
such that
$$H(t,\Phi_\omega^\infty(\omega t)(\xi_+,\eta_+))=H_\infty(\xi_+,\eta_+),\ t\in\R,\ (\xi_+,\eta_+)\in Y_{s'},\ \omega\in D_\varepsilon,$$
where $1\leq s'\leq \max\{s,1\}$.\\
\end{Theorem}
\begin{Remark}
In fact, $\Phi_\omega^\infty(\theta)$ and its inverse are bounded operators from $Y_{1}$ into itself for any $s\geq 0$.
\end{Remark}
We prove Theorem \ref{MainTheorem} in Section \ref{s4}.

\section{Application to the Quantum Harmonic Oscillator--Proof of  Theorem \ref{quantumth}}
In this section we prove Theorem \ref{quantumth} as a corollary of Theorem \ref{MainTheorem}.
\subsection{Verification of the hypothesis}
\begin{Lemma}\label{aspt}
When $\lambda_a=w_a,\ a\in\mathcal{E},$ Hypothesis $\mathrm{A1}$ holds true with $c_0=c_1= c_2=1$.
\end{Lemma}

As \cite{GP},
\begin{Lemma}\label{aspt02}
When $\lambda_a=w_a,\ a\in\mathcal{E},$ Hypothesis $\mathrm{A2}$ holds true with $D_0=[0,2\pi]^n$, $ \alpha_1=n+1,\ \alpha_2=1,\  \ c_3= c(n )$ and
$$\ D':=\{\omega\in[0,2\pi]^n\big|\ |\la k,\omega\ra+j|\geq\kappa(1+|j|), {\rm\ for\ all\ j\in\Z\ and\ k\in\Z^n\setminus\{0\}}\}.$$
\end{Lemma}
\begin{proof}
 Since $w_a-w_b\in\Z$, it is obtained by a straightforward computation.
\end{proof}

\begin{Lemma}\label{L3.3}
Let $d\geq 1$. Suppose that the potential $V: \T^n\times \R^d\ni (\theta, x)\rightarrow \R$ is $(s,\beta)-$admissible. Then there exists $\alpha=\alpha(d,s)>0$(see (\ref{alpha})) such that
the matrix function $P(\theta)$ defined by
$$(P(\theta))_{a}^{b}=\int_{\R^d}V(\theta, x)\Phi_a(x)\Phi_b(x)dx,\qquad a,b \in \mathcal{E}, $$
belongs to ${\mathcal{C}^{\beta}}(\T^n, \mathcal{M}_{s,\alpha})$.
\end{Lemma}
\begin{proof}We divide the proof into several steps.\\
\indent(a)\quad We show that
\begin{equation*}\label{3.3.1}
\sup_{\theta\in\R^n}|P(\theta)|_{s,\alpha}\leq C\left(d,s\right)\|V(\theta,\cdot)\|_{{\mathcal{C}^{\beta}}(\T^n, \mathcal{H}^s)}.
\end{equation*}
Recall that for $a,b\in \mathcal{E}$, $$(P(\theta))_{a}^{b}=\int_{\R^d}V(\theta, x)\Phi_a(x)\Phi_b(x)dx. $$
To estimate $|P(\theta)|_{s,\alpha}$ by definition, we turn to estimate
$$ \|P_{[a]}^{[b]}(\theta)\|=\sup_{\substack
{\Psi_a\left(x\right)\in E_{w_a},\ \Psi_b\left(x\right)\in E_{w_b},\\ {\|\Psi_a\left(x\right)\|_{L^2(\R^d)}}=\|\Psi_b\left(x\right)\|_{L^2(\R^d)}=1}}\int_{\R^d}V(\theta, x)\Psi_a(x)\Psi_b(x)dx. $$
From a similar proof in \cite{GP}(Lemma 3.2),
\begin{eqnarray*}
&&\left|\int_{\R^d} V \left(\theta,x\right) \Psi_a\left(x\right)\Psi_b\left(x\right)dx\right|\\
&\leq&\frac{C\left(d,s\right) }{\left(w_aw_b\right)^{\widetilde{\alpha}\left(p\right)/2}}\left(\frac{\sqrt{\min\left(w_a,w_b\right)}}{\sqrt{\min\left(w_a,w_b\right)}+|w_a-w_b|}\right)^{s/2}\| V\left(\theta,\cdot\right) \|_{\mathcal{H}^s\left(\R^d\right)}
\end{eqnarray*}
 where $\widetilde{\alpha}\left(p\right)= \frac{1}{12}$ if $d=1$ and $\widetilde{\alpha}\left(p\right)= \frac{1}{3p}\left(p\geq\frac{10}{3}\right) $ if $d=2$ and $\widetilde{\alpha}\left(p\right)=\frac{1}{2}\left(\frac{d}{3p}-\frac{d-2}{6}\right)>0$ if $d>2$ and $\frac{2(d+3)}{d+1}< p<\frac{2d}{d-2}$.
Set
\begin{equation}\label{alpha}
\alpha : = \frac{\widetilde{\alpha}\left(p\right)}{2}>0.
\end{equation}
It follows for $\theta\in\R^n$,
\begin{equation*}\label{}
 |P(\theta)|_{s,\alpha}\leq C\left(d,s\right)\|V(\theta,\cdot)\|_{ \mathcal{H}^s}\leq C\left(d,s\right)\|V(\theta,\cdot)\|_{{\mathcal{C}^{\beta}}(\T^n, \mathcal{H}^s)}.
\end{equation*}
(b)\quad  We show that $P(z)\in \mathcal{C}^0(\R^n, \mathcal{M}_{s,\alpha}).$
For any $z_1,z_2\in \R^n,\ a,b\in \mathcal{E}$,
$$(P(z_1)-P(z_2))_{a}^{b}=\int_{\R^d}(V(z_1, x)-V(z_2, x))\Phi_a(x)\Phi_b(x)dx. $$
A similar discussion as above tells us
\begin{equation*}\label{3.3.2}
 |P(z_1)-P(z_2)|_{s,\alpha}\leq C\left(d,s\right)\|V(z_1, \cdot)-V(z_2,\cdot)\|_{ \mathcal{H}^s}.
\end{equation*}
It follows that $ |P(z_1)-P(z_2)|_{s,\alpha}\rightarrow0$ as $z_1\rightarrow z_2$ by $V(z, \cdot)\in{\mathcal{C}^{\beta}}(\R^n, \mathcal{H}^s(\R^d)).$\\
(c)\quad We show that $P(z)$ is Fr\'echet differentiable at each $z_0\in\R^n$ and  for $h\in\R^n,\ a,b\in \mathcal{E}$,
$$(P'(z_0)h)_{a}^{b}=\int_{\R^d}\la V'_z(z_0, x),\ h\ra\Phi_a(x)\Phi_b(x)dx  $$ satisfying
\begin{equation}\label{3.3.3}
 \|P'(z_0)\|_{\mathfrak{L}(\R^n, \mathcal{M}_{s,\alpha})}\leq C(d,s)\|V(z,\cdot)\|_{{\mathcal{C}^{\beta}}(\R^n, \mathcal{H}^s)}.
\end{equation}
In fact, we define for $h\in\R^n,\ a,b\in \mathcal{E}$,
$$(\mathcal{A}h)_{a}^{b}:=\int_{\R^d}\la V'_z(z_0, x),\ h\ra\Phi_a(x)\Phi_b(x)dx.  $$
Clearly,  $\mathcal{A}$ is a linear map on $\R^n$. Since $\la V'_z(z_0, \cdot),\ h\ra\in \mathcal{H}^s(\R^d)$, from a similar discussion in (a) we obtain
\begin{equation*}\label{}
 |\mathcal{A}h|_{s,\alpha}\leq C\left(d,s\right)\| V'_z(z_0, \cdot)\|_{ \mathfrak{L}(\R^n, \mathcal{H}^s)}\|h\|\leq C\left(d,s\right)\|V(\theta,\cdot)\|_{{\mathcal{C}^{\beta}}(\R^n, \mathcal{H}^s)}\|h\|.
\end{equation*}
It follows that
\begin{equation*}\label{3.3.4}
 \|\mathcal{A}\|_{\mathfrak{L}(\R^n, \mathcal{M}_{s,\alpha})}\leq C(d,s)\|V(z,\cdot)\|_{{\mathcal{C}^{\beta}}(\R^n, \mathcal{H}^s)}.
\end{equation*}
\indent Next we show that
\begin{equation}\label{3.3.5}
 |P(z)-P(z_0)-\mathcal{A}(z-z_0)|_{s,\alpha}=o(|z-z_0|),\ z\rightarrow z_0.
\end{equation}
For $ a,b\in \mathcal{E}$,
$$(P(z)-P(z_0)-\mathcal{A}(z-z_0))_{a}^{b}=\int_{\R^d}\left(V(z,x)-V(z_0,x)- V'_z(z_0, x)(z-z_0)\right)\Phi_a(x)\Phi_b(x)dx.  $$
Note  $V(z, \cdot)\in{\mathcal{C}^{\beta}}(\R^n, \mathcal{H}^s(\R^d)),$ it follows from Taylor's theorem
\begin{equation}\label{3.3.6}
 \|V(z,\cdot)-V(z_0,\cdot)- V'_z(z_0, \cdot)(z-z_0)\|_{ \mathcal{H}^s}=o(|z-z_0|),\ z\rightarrow z_0.
 \end{equation}
 From a similar discussion as above we obtain
 \begin{equation*}\label{}
 |P(z)-P(z_0)-\mathcal{A}(z-z_0)|_{s,\alpha}\leq C\left(d,s\right)\|V(z,\cdot)-V(z_0,\cdot)- V'_z(z_0, \cdot)(z-z_0)\|_{ \mathcal{H}^s}.
\end{equation*}
Combining with (\ref{3.3.6}) we have (\ref{3.3.5}). Thus $P(z)$ is Fr\'echet differentiable on $z=z_0$  and $P'(z_0)=\mathcal{A}$ which satisfies (\ref{3.3.3}).\\
(d)\quad We show that  $P(z)\in \mathcal{C}^1(\R^n, \mathcal{M}_{s,\alpha}).$
Note that for any $z_1,z_2\in \R^n,\ a,b\in \mathcal{E}$,
$$((P'(z_1)-P'(z_2))h)_{a}^{b}=\int_{\R^d}\la V'_z(z_1, x)-V'_z(z_2,x),h\ra\Phi_a(x)\Phi_b(x)dx. $$
Thus
\begin{eqnarray*}\label{3.3.7}
 |(P'(z_1)-P'(z_2))h|_{s,\alpha}&\leq& C\left(d,s\right)\|\la V'_z(z_1, \cdot)-V'_z(z_2,\cdot),h\ra\|_{ \mathcal{H}^s}\\
  &\leq& C\left(d,s\right)\|  V'_z(z_1, \cdot)-V'_z(z_2,\cdot)\|_{ \mathfrak{L}(\R^n, \mathcal{H}^s)}\|h\|.
\end{eqnarray*}
It follows that
\begin{equation*}\label{}
 \|P'(z_1)-P'(z_2)\|_{\mathfrak{L}(\R^n, \mathcal{M}_{s,\alpha})}\leq C(d,s)\| V'_z(z_1, \cdot)-V'_z(z_2,\cdot)\|_{\mathfrak{L}(\R^n, \mathcal{H}^s)}.
\end{equation*}
It follows that $ \| V'_z(z_1, \cdot)-V'_z(z_2,\cdot)\|_{\mathfrak{L}(\R^n, \mathcal{H}^s)}\rightarrow0$ as $z_1\rightarrow z_2$ by $V(z, \cdot)\in{\mathcal{C}^{\beta}}(\R^n, \mathcal{H}^s(\R^d))$ which means that $ \|P'(z_1)-P'(z_2)\|_{\mathfrak{L}(\R^n, \mathcal{M}_{s,\alpha})}\rightarrow0$ as $z_1\rightarrow z_2$.\\
(e)\quad Inductively, we assume that $P(z)\in \mathcal{C}^m(\R^n, \mathcal{M}_{s,\alpha}),$ $m\leq \beta-1$, with
$$\left(P^{(m)}(z) (h_1,\cdots,h_m)\right)_{a}^{b}=\int_{\R^d}V_z^{(m)}(z,x) (h_1,\cdots,h_m)\Phi_a(x)\Phi_b(x)dx $$ satisfying
\begin{equation}\label{3.3.8}
 \|P^{(m)}(z)\|_{\mathfrak{L}_m(\R^n, \mathcal{M}_{s,\alpha})}\leq C(d,s)\|V(z,\cdot)\|_{{\mathcal{C}^{\beta}}(\R^n, \mathcal{H}^s)},
\end{equation}
where $\mathfrak{L}_m(\R^n, \mathcal{M}_{s,\alpha})$ denotes the multi-linear operator space $\mathfrak{L}(\underbrace{\R^n\times\cdots\times\R^n}_{m}, \mathcal{M}_{s,\alpha}).$
Then we show that
$P(z)\in \mathcal{C}^{m+1}(\R^n, \mathcal{M}_{s,\alpha})$  with
$$\left(P^{(m+1)}(z) (h_1,\cdots,h_{m+1})\right)_{a}^{b}=\int_{\R^d}V_z^{(m+1)}(z,x) (h_1,\cdots,h_{m+1})\Phi_a(x)\Phi_b(x)dx$$
and (\ref{3.3.8}) holds for $m$ replaced by $m+1$.
We follow the method  in steps (c) and (d), and divide the proof into two parts ($e_1$) and ($e_2$) respectively.\\
($e_1$)\quad We show that $P^{\left(m\right)}\left(z\right)$ is Fr\'echet differentiable on $z=z_0$ and  for $a,b\in \mathcal{E}$,
$$\left(P^{\left(m+1\right)}\left(z_0\right)(h_1,\cdots,h_{m+1})\right)_{a}^{b}=\int_{\R^d}  V^{\left(m+1\right)}_z\left(z_0, x\right)(h_1,\cdots,h_{m+1})\Phi_a\left(x\right)\Phi_b\left(x\right)dx  $$ with
\begin{equation}\label{3.3.9}
 \|P^{\left(m+1\right)}\left(z_0\right)\|_{\mathfrak{L}_{m+1}\left(\R^n, \mathcal{M}_{s,\alpha}\right)}\leq C\left(d,s\right)\|V\left(z,\cdot\right)\|_{{\mathcal{C}^{\beta}}\left(\R^n, \mathcal{H}^s\right)}.
\end{equation}
In fact, given $z\in\R^n$, we define for $h_1,\cdots,h_{m+1}\in\R^n,\ a,b\in \mathcal{E}$,
$$\left(\mathcal{B}(z)\left(h_1,\cdots,h_{m+1}\right)\right)_{a}^{b}:=\int_{\R^d}  V^{\left(m+1\right)}_z\left(z, x\right)\left(h_1,\cdots,h_{m+1}\right)\Phi_a\left(x\right)\Phi_b\left(x\right)dx.  $$
Clearly,  $\mathcal{B}\in \mathfrak{L}_{m+1}\left(\R^n, \mathcal{M}_{s,\alpha}\right)$. Since $V^{\left(m+1\right)}_z\left(z, \cdot\right)\in \mathfrak{L}_{m+1}\left(\R^n, \mathcal{H}^s\right)$, $\mathcal{B}$ is clearly multi-linear and $ V^{\left(m+1\right)}_z\left(z, \cdot\right)\left(h_1,\cdots,h_{m+1}\right)\in \mathcal{H}^s\left(\R^d\right)$. Combining with
\begin{equation*}\label{}
 \sup_{z\in\R^n}\|V^{\left(m+1\right)}_z\left(z\right)\|_{\mathfrak{L_{m+1}}\left(\R^n, \mathcal{H}^s\right)}\leq  \|V\left(z,\cdot\right)\|_{{\mathcal{C}^{\beta}}\left(\R^n, \mathcal{H}^s\right)}
\end{equation*}
and a similar discussion as above we obtain
\begin{eqnarray*}\label{}
 |\mathcal{B}(z)\left(h_1,\cdots,h_{m+1}\right)|_{s,\alpha}&\leq& C\left(d,s\right)\|V^{\left(m+1\right)}_z\left(z\right)\left(h_1,\cdots,h_{m+1}\right)\|_{ \mathcal{H}^s}\\
 &\leq&  C\left(d,s\right)\|V\left(\theta,\cdot\right)\|_{\mathcal{C}^{\beta}\left(\R^n, \mathcal{H}^s\right)}\|h_{1}\|\cdots \|h_{m+1}\|.
\end{eqnarray*}
It follows that
\begin{equation*}\label{3.3.10}
 \|\mathcal{B}\|_{\mathfrak{L}_{m+1}\left(\R^n, \mathcal{M}_{s,\alpha}\right)}\leq C\left(d,s\right)\|V\left(z,\cdot\right)\|_{{\mathcal{C}^{\beta}}\left(\R^n, \mathcal{H}^s\right)}.
\end{equation*}
Set $B=\mathcal{B}\left(z_0\right)$.
Next we show that
\begin{equation*}\label{3.3.11}
 |P^{\left(m\right)}\left(z\right)-P^{\left(m\right)}\left(z_0\right)-B\left(z-z_0\right)|_{s,\alpha}=o\left(|z-z_0|\right),\ z\rightarrow z_0.
\end{equation*}
For $ a,b\in \mathcal{E}$,
\begin{eqnarray*}\label{}
 &&\left(\left(P^{\left(m\right)}\left(z\right)-P^{\left(m\right)}\left(z_0\right)-B\left(z-z_0\right)\right)\left(h_1,\cdots,h_{m}\right)\right)_{a}^{b}\\
 &=&\int_{\R^d}\left(V^{\left(m\right)}_z\left(z,x\right)-V^{\left(m\right)}_z\left(z_0,x\right)- V^{\left(m+1\right)}_z\left(z_0, x\right)\left(z-z_0\right)\right)\left(h_1,\cdots,h_{m}\right)\Phi_a\left(x\right)\Phi_b\left(x\right)dx.  \end{eqnarray*}
Note  $V\left(z, \cdot\right)\in{\mathcal{C}^{\beta}}\left(\R^n, \mathcal{H}^s\left(\R^d\right)\right),$ it follows from Taylor's theorem
\begin{equation*}\label{3.3.12}
 \|V^{\left(m\right)}_z\left(z,\cdot\right)-V^{\left(m\right)}_z\left(z_0,\cdot\right)- V^{\left(m+1\right)}_z\left(z_0, \cdot\right)\left(z-z_0\right)\|_{\mathfrak{L_{m+1}}\left(\R^n, \mathcal{H}^s\right)}=o\left(|z-z_0|\right),\ z\rightarrow z_0.
 \end{equation*}
 From a similar discussion as in (c) we obtain
 \begin{equation*}\label{}
 |\left(P^{\left(m\right)}\left(z\right)-P^{\left(m\right)}\left(z_0\right)-B\left(z-z_0\right)\right)\left(h_1,\cdots,h_{m}\right)|_{s,\alpha}\leq C\left(d,s\right)o\left(|z-z_0|\right)\|h_{1}\|\cdots \|h_{m}\|,\ z\rightarrow z_0.
\end{equation*}
Thus $P^{\left(m\right)}\left(z\right)$ is Fr\'echet differentiable  and $P^{\left(m+1\right)}\left(z \right)=\mathcal{B}$ which satisfies  (\ref{3.3.9}).\\
 ($e_2$)\quad Since $$ \| V^{\left(m+1\right)}_z\left(z_1, \cdot\right)-V^{\left(m+1\right)}_z\left(z_2,\cdot\right)\|_{\mathfrak{L}_{m+1}\left(\R^n, \mathcal{H}^s\right)}\rightarrow0\ {\rm as}\ z_1\rightarrow z_2$$  by $V\left(z, \cdot\right)\in{\mathcal{C}^{\beta}}\left(\R^n, \mathcal{H}^s\left(\R^d\right)\right)$. From a similar discussion in (d) we have
\begin{eqnarray}\label{3.3.13}
\nonumber && \|P^{\left(m+1\right)}\left(z_1\right)-P^{\left(m+1\right)}\left(z_2\right)\|_{\mathfrak{L}_{m+1}\left(\R^n, \mathcal{M}_{s,\alpha}\right)}\\
&\leq& C\left(d,s\right)\| V^{\left(m+1\right)}_z\left(z_1, \cdot\right)-V^{\left(m+1\right)}_z\left(z_2, \cdot\right)\|_{\mathfrak{L}_{m+1}\left(\R^n, \mathcal{H}^s\right)},
\end{eqnarray}
which means that $P\left(z\right)\in \mathcal{C}^{m+1}\left(\R^n, \mathcal{M}_{s,\alpha}\right)$ and
\begin{equation*}\label{}
 \|P^{\left(m+1\right)}\left(z\right)\|_{\mathfrak{L}_{m+1}\left(\R^n, \mathcal{M}_{s,\alpha}\right)}\leq C\left(d,s\right)\|V\left(z,\cdot\right)\|_{{\mathcal{C}^{\beta}}\left(\R^n, \mathcal{H}^s\right)}.
\end{equation*}
 (f)\quad Finally, consider $P\left(z\right)\in \mathcal{C}^{\beta}\left(\R^n, \mathcal{M}_{s,\alpha}\right)$. Denote   $b:=\beta-[\beta]$, we   show that for any  $1\leq l\leq[\beta]$ integer and $0<|z_1-z_2|<2\pi,$
\begin{equation}\label{3.3.14}
\frac{\|P^{\left(l\right)}\left(z_1\right)-P^{\left(l\right)}\left(z_2\right)\|_{\mathfrak{L}_{l}\left(\R^n, \mathcal{M}_{s,\alpha}\right)}}{{|z_1-z_2|^{b}}}\leq \|V\left(z,\cdot\right)\|_{{\mathcal{C}^{\beta}}\left(\R^n, \mathcal{H}^s\right)}.
\end{equation}
In fact, note that $V\left(z, \cdot\right)\in{\mathcal{C}^{\beta}}\left(\R^n, \mathcal{H}^s\left(\R^d\right)\right),$ from the definition we obtain that for $0<|z_1-z_2|<2\pi,$
$$\frac{\|V_z^{\left(l\right)}\left(z_1, \cdot\right)-V_z^{\left(l\right)}\left(z_2, \cdot\right)\|_{\mathfrak{L}_{l}\left(\R^n, \mathcal{M}_{s,\alpha}\right)}}{{|z_1-z_2|^{b}}}\leq \|V\left(z,\cdot\right)\|_{{\mathcal{C}^{\beta}}\left(\R^n, \mathcal{H}^s\right)}.$$
Thus (\ref{3.3.14}) holds  by (\ref{3.3.13}) for any $1\leq l\leq[\beta]$ which means that $P\left(z\right)\in \mathcal{C}^{\beta}\left(\R^n, \mathcal{M}_{s,\alpha}\right)$ and
$$\|P\left(\theta\right)\|_{{\mathcal{C}^{\beta}}\left(\T^n, \mathcal{M}_{s,\alpha}\right)}\leq C\left(d,s\right)\|V\left(z,\cdot\right)\|_{{\mathcal{C}^{\beta}}\left(\R^n, \mathcal{H}^s\right)}. $$
\end{proof}

\subsection{Proof of Theorem \ref{quantumth}}
Following the discussions stated in Subsection \ref{s1.1},
the Schr\"odinger equation (\ref{HOeq}) is equivalent to Hamiltonian system with (\ref{HOfun}). Expanding it on the Hermite basis $(\Phi_a)_{a\in\mathcal{E}}$,
it is equivalent to the system   governed by (\ref{hameq000})
 which reads as (\ref{hs00}), a special case of system (\ref{hameq00}) with $\lambda_a=w_a$ and $P(\omega t)$ satisfying (\ref{Pijform}). By
 Lemmas given above, if $V$ is $(s,\beta)-$admissable, we can apply Theorem \ref{MainTheorem}  to
  (\ref{hs00})  with $\gamma_1= n+d+2,\
\gamma_2=\frac{\alpha }{4+d+2\alpha }$ and $\alpha$ given by (\ref{alpha}),   this leads to Theorem \ref{quantumth}.\\
 \indent More precisely, in the new coordinates given by Theorem \ref{MainTheorem}, $(\xi,\eta)=(\overline{ {M}}_\omega\xi_+, {M}_\omega\eta_+)$,
 system (\ref{hs00}) becomes autonomous and decomposes into blocks as follows:
 \begin{eqnarray*}
\left\{\begin{array}{c}
\dot{\xi}_{+,[a]}=-\mathrm{i} (\overline{{N}}_{\infty})_{[a]}\xi_{+,{[a]}},\\
\dot{\eta}_{+,[a]}= \ \ \mathrm{i}(N_{\infty})_{[a]}\eta_{+,{[a]}},
\end{array}\right.\ \ \ a\in \mathcal{E},\label{inftyeq}
\end{eqnarray*}
where    $N_\infty=N_\infty(\omega)\in\mathcal{NF}$. Hence the solution start from $(\xi_+(0),\eta_+(0))$ is given by  $$(\xi_+(t),\eta_+(t))=(e^{-\mathrm{i}{t\overline{N}}_\infty }\xi_+(0),e^{ \mathrm{i}{t {N}}_\infty }\eta_+(0)),\ t\in\R.$$
Then the solution $u(t,x)$ of (\ref{HOeq}) corresponding to the initial data $u_0(x)=\sum_{a\in\mathcal{E}}\xi_a(0)\Phi_a(x)\in \mathcal{H}^{s'}$ with $1\leq s'\leq \max\{s,1\}$  is formulated by
$u(t,x)=\sum_{a\in\mathcal{E}}\xi_a(t)\Phi_a(x)$ with
\begin{eqnarray*}\label{xit}
\xi(t)=\overline{M}_\omega(\omega t)e^{-\mathrm{i}{t\overline{N}}_\infty }M^{T}_\omega(0)\xi(0),
\end{eqnarray*}
where we use the fact $(\overline{ {M}}_\omega)^{-1}=M^{T}_\omega.$\\
\indent In other words, let us define the transformation $\Psi_\omega(\theta)\in\mathfrak{L}(\mathcal{H}^{s'}),\ 0\leq s'\leq s,$ by
$$\Psi_\omega(\theta)(\sum_{a\in\mathcal{E}}\xi_a\Phi_a(x)):=\sum_{a\in\mathcal{E}}(M^T_\omega(\theta)\xi)_a\Phi_a(x)=\sum_{a\in\mathcal{E}}\xi_{+,a}\Phi_a(x).$$
From a straightforward computation(the proof is given in the Appendix), we have
\begin{Lemma}\label{psismooth}
For $0\leq s'\leq s$ and $\alpha>0$ given by (\ref{alpha}),
$$\|\Psi_\omega(\theta)-id\|_{\mathcal{C}^{\mu}(\T^n, \mathfrak{L}(\mathcal{H}^{s'},\mathcal{H}^{s'+2\alpha}))}\leq C \varepsilon^{\frac{3}{2\beta}(\frac{2}{9}\beta-\mu)},$$
where $\mu$ is defined in Theorem \ref{MainTheorem}.
\end{Lemma}
\indent  Moreover, $u(t,x)$ satisfies (\ref{HOeq}) if and only if  $v(t,x)=\Psi_\omega(\omega t)u(t,x)$ satisfies the autonomous equation:
$$i\partial_t v+(-\partial_{xx}+|x|^2)v+\varepsilon Wv=0,$$
where  $W(\sum_{a\in\mathcal{E}}\xi_{a}\Phi_a(x))=\frac{1}{\varepsilon}\sum_{a\in\mathcal{E}}((N_\infty-N_0)^T\xi)_a\Phi_a(x)$.
 Denote by $(W_a^b)_{a,b\in\mathcal{E}}$ the infinite matrix
 of the operator $W$ written in the Hermite basis($W_a^b=\int_{\R^d} W\Phi_a\Phi_bdx$), then $W $ is block diagonal. Denote
$\la V \ra(x):=\frac{1}{(2\pi)^n}\int_{\T^n}V(\theta,x)d\theta$ the mean value of $V$ on the torus with $(\la V \ra_a^b)_{a,b\in\mathcal{E}}$ the corresponding infinite matrix where
$$\la V \ra_a^b= \int_{\R^d} \la V \ra \Phi_a\Phi_b  dx =\frac{1}{(2\pi)^n}\int_{\T^n}P_a^b(\theta)d\theta=(P_a^b)^0$$
with $P_a^b(\theta)=\sum_{k\in\Z^n}(P_a^b)^k e^{ik\theta}.$  From (\ref{ntilde}) with $\nu=0$,  $ (\widetilde{N}_0)_{[a]} =\frac{\varepsilon}{(2\pi)^n}\int_{\T^n}  P^{(1)}  _{[a]}(\theta)d\theta.$ It holds that
 \begin{eqnarray*}\label{}
\|(\widetilde{N}_0)_{[a]}-\varepsilon P_{[a]}^0\|&=& \frac{\varepsilon}{(2\pi)^n}\left\|\int_{\T^n}(P^{(1)}(\theta)-P(\theta))_{[a]}d\theta\right\|\\
&\leq& c\varepsilon w_a^{-2\alpha}[P^{(1)}(\theta)-P(\theta)]^{\T^n,D_\varepsilon}_{s,\alpha}\\
&\leq& c\varepsilon w_a^{-2\alpha}{\sigma_1}^{\beta}= cw_a^{-2\alpha}\varepsilon \epsilon_1
\end{eqnarray*}
by Lemma \ref{L4.7}. On the other hand, from Lemma \ref{convergence01}, $[N_\infty-N_1]^{\T^n,D_\varepsilon}_{s,\alpha}\leq   2\epsilon_1.$
Thus,
$$w_a^{ 2\alpha}\|(W-\la V \ra)_{[a]}\|\leq \frac{w_{a}^{2\alpha}}{\varepsilon}\|(N_\infty-N_0 -\widetilde{N}_{0 })_{[a]}\|+\frac{w_{a}^{2\alpha}}{\varepsilon}\|(\widetilde{N}_{0 }-\varepsilon P ^0)_{[a]}\| \leq c  \varepsilon^{\frac12}. $$
Therefore,
$$\|(W_{a}^b)_{a,b\in\mathcal{E}}-\Pi(\la V \ra_{a}^b)_{a,b\in\mathcal{E}}\|_{\mathfrak{L}(\ell_{s'}^2)}\leq c\varepsilon^{\frac12}$$
 for $0\leq s'\leq s$, where $\Pi$ is the projection on the diagonal blocks.\\
\indent The proofs of Corollary \ref{coro01} and  \ref{coro02} are similar as \cite{GP}, we omit it for simplicity.\\

 \section{Proof of  Theorem \ref{MainTheorem}}\label{s4}
\indent In this section we will use a universal constant $C$ to simplify the proof, which depends on $n,\beta,d,s$ and is changing in the context. \\
\indent The system (\ref{hameq00})  is equivalent to the autonomous system:
\begin{eqnarray}
\left\{\begin{array}{c}
\dot{\xi}_a=-\mathrm{i}\lambda_a\xi_a-\mathrm{i}\varepsilon (P^T(\theta)\xi)_a,\\
\dot{\eta}_a=\ \  \mathrm{i}\lambda_a\eta_a+\mathrm{i}\varepsilon (P(\theta)\eta)_a,\ \ \\
\dot{y}= - \varepsilon\la\xi,   \nabla_{\theta} P(\theta)\eta\ra, \\
\dot{\theta}=\ \omega,\qquad\qquad\qquad\qquad\
\end{array} a\in\mathcal{E}\right.\label{autohs}
\end{eqnarray}
with
 the Hamiltonian
\begin{equation}\label{autoH}
\mathcal{H}(\theta,y,\xi,\eta,\omega)= \sum\limits_{j=1}^n\omega_jy_j+\sum\limits_{a\in\mathcal{E}} \lambda_a\xi_a\eta_a+ \varepsilon\la\xi,  P(\theta)\eta\ra
\end{equation}
in the extended phase space $\mathcal P_s:=\T^n\times \R^{n}\times Y_s$,
 and $\lambda_{a} $   satisfies  Hypothesis $\mathrm{A1}-\mathrm{A2}$.

\subsection{Analytic approximation to a ${\mathcal{C}^{\beta}}$ smooth   function} \indent The ${\mathcal{C}^{\beta}}$ smooth Hamiltonian function (\ref{autoH}) can be approximated by a series of  \emph{analytic} Hamiltonians
\begin{equation*}\label{autoHnu}
H^{(\nu)}(\theta,y,\xi,\eta,\omega)= \sum\limits_{j=1}^n\omega_jy_j+\sum\limits_{a\in\mathcal{E}} \lambda_a\xi_a\eta_a+ \varepsilon\la\xi, P^{(\nu)}(\theta)\eta\ra,\  \nu=1,2,\cdots,
\end{equation*}
where $P^{(\nu)}(\theta)$ will be given in the following. In order to extend the $\mathcal{C}^{\beta}$ function to a complex neighborhood of $\T^n$, we need
the famous results. \\
\begin{Lemma}(Jackson, Moser and Zehnder)\label{smoothing}
Let $X$ be a complex Banach space and  $f\in {\mathcal{C}^{\beta}}(\R^n, X)$ for some $\beta>0$ with finite ${\mathcal{C}^{\beta}}(\R^n, X)$ norm. Let $\phi$ be a radical - symmetric, $\mathcal{C}^{\infty}$ function, having as support the closure of the unit ball
centered at the origin, where $\phi$ is completely flat and take value $1$, let $K=\hat{\phi}$ be its Fourier transform. For all $\sigma>0$ define
$$S_{\sigma}f(z,\cdot)=K_{\sigma}\ast f=\frac{1}{\sigma^n}\int_{\R^n}K(\frac{z-y}{\sigma})f(y, \cdot)dy.$$
Then there exists a constant $C>0$ depending on $\beta, n$ and $X$ such that the following holds:
for any $\sigma>0$, the function $f_{\sigma}(z)$ is a real analytic function from $\C^n$ to $X$ such that if
$$\Delta_{\sigma}^{n} : = \left\{z\in \C^n\left| \right.  |\Im z_j|\leq \sigma, 1\leq j\leq n\right\}, $$
then for any $\nu\in \N^n$ such that $|\nu|\leq \beta$ one has
\begin{eqnarray*}\label{smoothing1.1}
\sup\limits_{z\in \Delta_{\sigma}^n}\left\|\partial^{\nu}f_{\sigma}(z)-\sum\limits_{|k|\leq \beta-|\nu|}\frac{\partial^{k+\nu}f(\Re z)}{k!}({\rm i}\Im z)^{k}\right\|_{X_{\nu}}\leq C\|f\|_{{\mathcal{C}^{\beta}}(\R^n, X)}\sigma^{\beta-|\nu|},
\end{eqnarray*}
and for all $0\leq \sigma_1\leq \sigma$,
$$\sup\limits_{z\in \Delta_{\sigma_1}^n}\|\partial^{\nu}f_{\sigma}(z)-\partial^{\nu}f_{\sigma_1}(z)\|_{X_{\nu}}\leq C\|f\|_{{\mathcal{C}^{\beta}}(\R^n, X)}\sigma^{\beta-|\nu|}.$$
The function $f_{\sigma}$ preserves periodicity(i.e. if $f$ is $T-$periodic in any of its variables $z_j$, so is $f_{\sigma}$).
\end{Lemma}
The same theorem was  also used in  \cite{Bam1},  \cite{YZ13}, etc. \\
\indent The converse statement of Lemma \ref{smoothing} holds only if $\mu$ is not an integer. A
classical version of this converse result is due to Bernstein and relates the differentiability
properties of a periodic function to quantitative estimates for an
approximating sequence of trigonometric polynomials. In the following lemma, we suppose $X$ be a complex Banach space as above.
\begin{Lemma}\label{smoothinginverse}Let $\ell\geq0$ and $n$ be a positive integer. Then there exists a constant $c=c(\ell,n)$ such that
if $f:\R^n\rightarrow X$ is the limit of a sequence of real analytic maps $f_\nu(x)$ in the strips $|\Im
x|\leq\sigma_\nu:=\sigma^{(\frac{3}{2})^\nu}$ with $0<\sigma\leq\frac{1}{4}$ and
\begin{eqnarray*}
f_0=0,\quad|f_\nu-f_{\nu-1}|_{X}\leq c\sigma_\nu^\ell
\end{eqnarray*}
for $|\Im x|\leq\sigma_\nu$, $\nu=1,2,\cdots,$ then $f\in \mathcal{C}^\mu(\R^n,X)$ for every $\mu\leq\ell$ which is not an integer and
\begin{eqnarray*}
|f|_{\mathcal{C}^{\mu}(\R^n, X)}\leq \frac{4c(\ell,n)}{\iota(1-\iota)}\sigma^{\ell-\mu},\ 0<\iota:=\mu-[\mu]<1.
\end{eqnarray*}
\end{Lemma}
For the proof see section \ref{appendix}.\\
\indent For our applications we choose $f(z) =  P(z)$, $X=(\mathcal{M}_{s,\alpha},|\cdot |_{s,\alpha})$. From Lemma \ref{smoothing}
we denote
\begin{equation}\label{moguangpz}
S_{\sigma}P(z) = \sigma^{-n}\int_{\R^n}K(\frac{z-y}{\sigma})P(y)dy,\ z\in \Delta_{\sigma}^{n}.
\end{equation}
From Lemma \ref{smoothing} again we have that for any $\sigma>0$, $S_{\sigma}P(z)$ is a real analytic function from $\C^n$ to $\mathcal{M}_{s,\alpha}$ such that
  for any $k\in \N^n$ satisfying $|k|\leq \beta$, one has
\begin{eqnarray}\label{smoothing1.1}
\sup\limits_{z\in \Delta_{\sigma}^n}\|\partial^{k}S_{\sigma}P(z)-\sum\limits_{|m|\leq \beta-|k|}\frac{\partial^{m+k}P(\Re z)}{m!}({\rm i}\Im z)^{m}\|_{X_{k}}\leq C\|P(\theta)\|_{{\mathcal{C}^{\beta}}(\T^n, \mathcal{M}_{s,\alpha})}\sigma^{\beta-|k|},
\end{eqnarray}
and for all $0\leq \sigma'\leq \sigma $,
\begin{eqnarray}\label{smoothing1.2}
\sup\limits_{z\in \Delta_{\sigma'}^n}\|\partial^{k}S_{\sigma }P(z)-\partial^{k}S_{\sigma'}P(z)\|_{X_{k}}\leq C\|P(\theta)\|_{{\mathcal{C}^{\beta}}(\T^n, \mathcal{M}_{s,\alpha})}\sigma ^{\beta-|k|}.
\end{eqnarray}
\begin{Remark}
Since $S_{\sigma}P(z)$ preserves periodicity, we often  write $S_{\sigma}P(\theta)$ instead of $S_{\sigma}P(z)$.  Recall $P(\theta)$ is Hermitian and, from \cite{Sal04}, $K(\R^n)\in \R$,  then $S_{\sigma}P(\theta)$ is also Hermitian when $\theta\in \T^n$ by (\ref{moguangpz}).
\end{Remark}
Suppose $0<\cdots<\sigma_{\nu}<\cdots<\sigma_{1}<\sigma_{0}.$ Then we can construct a series of analytic functions $\{S_{\sigma_\nu}P(\theta),\ \theta\in \T_{\sigma_\nu}^{n}\}_{\nu\in\N}$. From (\ref{smoothing1.2}) we have
 \begin{Lemma}\label{PP001}
 For $ |\Im \theta|\leq\sigma_{\nu }$, $\nu=1,2,\cdots,$
\begin{eqnarray*}
|S_{\sigma_{\nu }}P (\theta)-S_{\sigma_{\nu-1}}P(\theta)|_{s,\alpha}\leq
C \sigma_{\nu-1}^{\beta}\|P(\theta)\|_{\mathcal{C}^{\beta}(\T^n, \mathcal{M}_{s,\alpha})}\leq C \sigma_{\nu-1}^{\beta} .
\end{eqnarray*}
\end{Lemma}

\begin{Lemma}\label{PP002}
For $ |\Im \theta|\leq\sigma_{0}\leq 1$,
\begin{eqnarray}\label{guanghua2-1}
|S_{\sigma_{0}}P (\theta ) |_{s,\alpha}\leq
C.
\end{eqnarray}
\end{Lemma}
\begin{proof}
From (\ref{smoothing1.1}), for  $|\Im \theta|\leq\sigma_0\leq1$ we have
$
|S_{\sigma_0}P(\theta)-\sum\limits_{|k|\leq \beta}\frac{\partial^{k}P(\Re \theta)}{k!}({\rm i}\Im \theta)^{k}|_{s,\alpha}\leq
C.
$
Thus,
$
|\sum\limits_{|k|\leq \beta}\frac{\partial^{k}P(\Re \theta)}{k!}({\rm i}\Im \theta)^{k}|_{s,\alpha}\leq \sum_{|k|\leq \beta}|\partial^{k} P(\Re \theta)|_{s,\alpha}|\sigma_0^{|k|}\leq C.
$
Then we obtain (\ref{guanghua2-1}).
\end{proof}

\begin{Lemma}\label{L4.7}
For $\theta\in \T^n$, $|S_{\sigma_{\nu}}P(\theta)-P(\theta)|_{s,\alpha}\leq C\sigma_{\nu}^{\beta}$,\ $\nu=0,1,2,\cdots.$
\end{Lemma}
\begin{proof}
From Lemma \ref{smoothing}, we have for $|\Im z|\leq \sigma_{\nu}$,
\begin{eqnarray*}
|S_{\sigma_{\nu}}P(z)-\sum\limits_{|k|\leq \beta}\frac{\partial^{k}P(\Re z)}{k!}(i \Im z)^{k}|_{s,\alpha}\leq C\|P(\theta)\|_{{\mathcal{C}^{\beta}}(\T^n, \mathcal{M}_{s,\alpha})}\sigma_{\nu}^{\beta}\leq C\sigma_{\nu}^{\beta}.
\end{eqnarray*}
On the other hand, if $z\in \R^d$,
$\sum\limits_{0<|k|\leq \beta}\frac{\partial^{k}P(\Re z)}{k!}(i \Im z)^{k}=0$.
Thus for $\theta\in \T^d$ it follows
$|S_{\sigma_{\nu}}P(\theta)-P(\theta)|_{s,\alpha}\leq C\sigma_{\nu}^{\beta}$.
\end{proof}
In the following we will write $P^{(\nu)}(\theta) : = S_{\sigma_{\nu}}P(\theta)$ for simplicity. Combining with all the above lemmas, we have
\begin{Lemma}\label{PP}
\begin{eqnarray}
\nonumber
[P(\theta)]_{s,\alpha}^{D_0}&\leq&  C, \qquad \theta\in\T^n;\\ \nonumber
[P^{(0)}(\theta)]_{s,\alpha}^{D_0,\sigma_{0}}&\leq&  C;\\ \label{guanghua4}
[P^{(\nu)} (\theta)-P(\theta)]_{s,\alpha}^{D_0}&\leq&
C\sigma_\nu^{\beta},    \quad \theta\in \T^n;\\ \nonumber
[P^{(\nu+1)} (\theta)-P^{(\nu)} (\theta)]_{s,\alpha}^{D_0,\sigma_{\nu+1}}&\leq&  C\sigma_\nu^{\beta}.
\end{eqnarray}
\end{Lemma}
\begin{Lemma}\label{l4.1}
For $\theta\in \T^n$,
\begin{eqnarray}\label{41two}
P(\theta)=P^{(0)}(\theta)+\sum_{\nu=0}^\infty(P ^{(\nu+1)}(\theta)-P ^{(\nu)}(\theta)),
\end{eqnarray}
under the assumption
\begin{eqnarray}\label{conditionsection4}
{\rm B_1}.\qquad \sum\limits_{\nu=0}^{\infty}\sigma_{\nu}^{\beta}<\infty.
\end{eqnarray}
\end{Lemma}
\begin{proof}
From Lemma \ref{PP} and  (\ref{conditionsection4}),
\begin{eqnarray*}
|P^{(0)}(\theta)+\sum_{\nu=0}^\infty(P ^{(\nu+1)}(\theta)-P ^{(\nu)}(\theta))|_{s,\alpha}\leq
|P^{(0)}(\theta)|_{s,\alpha}+\sum_{\nu=0}^\infty|P ^{(\nu+1)}(\theta)-P ^{(\nu)}(\theta)|_{s,\alpha}\leq C
\end{eqnarray*}
for any $\theta\in \T^n$.
Thus for $\theta\in \T^n$, $P^{(0)}(\theta)+\sum_{\nu=0}^\infty(P ^{(\nu+1)}(\theta)-P ^{(\nu)}(\theta))$ is an element in
$(\mathcal{M}_{s,\alpha}, |\cdot |_{s,\alpha})$. From Lemma \ref{PP}, $P(\theta)$ is also an element in
$(\mathcal{M}_{s,\alpha}, |\cdot |_{s,\alpha})$ for $\theta\in \T^n$.
The equality (\ref{41two}) follows from  (\ref{guanghua4}) and (\ref{conditionsection4}).
\end{proof}

\subsection{Homological equation} The proof of Theorem \ref{MainTheorem} is followed by an iterative KAM procedure  where in each step we will consider a
homological equation of the form
\begin{eqnarray*}\label{homoeq}
\omega\cdot\nabla_\theta F(\theta,\omega)-\mathrm{i}[N(\omega),F(\theta,\omega)]+Q(\theta,\omega)=\mathrm{remainder}
\end{eqnarray*}
with $N(\omega), \omega\in D\subset D_0,$ in normal form  close to $N_0=diag(\lambda_a)_{a\in\mathcal{E}}$ and $Q\in \mathcal{M}_{s,\alpha}(D,\sigma).$ We can construct a solution
$F\in \mathcal{M}_{s,\alpha}^+(D',\sigma')$, $0<\sigma'<\sigma$,
$D' \subset D$ as in  \cite{GP}.
\begin{Proposition}[\cite{GP}, Proposition 4.1]\label{pro4.1}
Let $(\theta,\omega)\in \T^n_\sigma\times D$ with $0<\sigma\leq1,\ D\subset D_0$. Suppose $D\ni\omega\mapsto N(\omega)\in \mathcal{NF}$ be a $\mathcal{C}^1$ mapping that satisfies $[N-N_0]_{s,\alpha}^{D}\leq
\frac{c_0}{4}$, $Q\in \mathcal{M}_{s,\alpha}(D,\sigma)$ Hermitian, $0<\kappa<\frac{c_0}{2}$ and $K\geq1$. Then for any $0<\sigma'<\sigma$ there exists a subset
$D'=D'(\kappa,K)\subset D$, satisfying ${\rm Meas}(D\setminus D')\leq c(n,c_0,\alpha,\alpha_2)K^{\gamma_1}\kappa^{\gamma_2}$ with
$\gamma_1=\max\{d+n+2,\alpha_1\},\ \gamma_2=\frac{\alpha\alpha_2}{4+d+2\alpha\alpha_2}$, and there exist $\widetilde{N}\in
\mathcal{M}_{s,\alpha}(D')\cap\mathcal{NF},$ $F\in \mathcal{M}_{s,\alpha}^+{(D',\sigma')}$  and $R\in \mathcal{M}_{s,\alpha}{(D',\sigma')}$, $\mathcal{C}^1$ in $\omega$ and analytic in $\theta$, such that
\begin{eqnarray*}\label{homoeq01}
\omega\cdot\nabla_\theta F(\theta,\omega)-\mathrm{i}[N(\omega),F(\theta,\omega)]=\widetilde{N}(\omega)-Q(\theta,\omega)+R(\theta,\omega)
\end{eqnarray*}
for all  $(\theta,\omega)\in \T^n_{\sigma'}\times D'$ and
\begin{eqnarray*}
 \label{homoF}
[F]_{s,\alpha+}^{D',\sigma'}&\leq&  \frac{c(n,d,s,\alpha)K^{1+d}}{\kappa^{2+d/\alpha}(\sigma -\sigma')^{n}}
[Q]_{s,\alpha}^{D,\sigma};\\ \label{homoN}
[\widetilde{N}]_{s,\alpha}^{D'}&\leq&  [Q]_{s,\alpha}^{D,\sigma};\\ \label{homoR}
[R]_{s,\alpha}^{D',\sigma'}&\leq&  \frac{c(n,d,s,\alpha)K^{1+d/2}e^{-(\sigma-\sigma')K/2}}{\kappa^{1+d/(2\alpha)}(\sigma-\sigma')^n}
[Q]_{s,\alpha}^{D,\sigma}.
\end{eqnarray*}
Moreover, $\widetilde{N},$ $F(\theta) $  and $R(\theta) $ are Hermitian when $\theta\in\T^n$.
\end{Proposition}

\subsection{The KAM Step.}\label{KAMstep}
 As the KAM proof in \cite{CQ,YZ13}, we begin with the initial Hamiltonian
 $H^{(0)}=h+ q_0$ with $h=\sum\limits_{j=1}^n\omega_jy_j+\sum\limits_{a\in\mathcal{E}}\lambda_a\xi_a\eta_a:=\la\omega,y\ra+\la\xi, N_0\eta\ra$, and $q_0= \la \xi, Q_0
 (\theta)\eta\ra$ with $Q_0=\varepsilon  P^{(0)}\in \mathcal{M}_{s,\alpha} (D_0,\sigma_0) $. For simplicity we set $\sigma_0=1$,
by Lemma \ref{PP}, \begin{eqnarray*}\label{}
[Q_0]_{s,\alpha}^{D_0,\sigma_{0}}=[\varepsilon P^{(0)}(\theta)]_{s,\alpha}^{D_0,\sigma_{0}}\leq C\varepsilon : = \frac12 \epsilon_0.
\end{eqnarray*}
\indent In the $\nu$th step of the KAM scheme,  we consider the Hamiltonian
 $H^{(\nu)}=h+ q_\nu$   defined on  $\T^n_{\sigma_\nu}\times\R^n\times Y_s\times  D_{0}$,   where $q_\nu=\la \xi,Q_\nu
 (\theta)\eta\ra$ with $Q_\nu=\varepsilon P^{(\nu)}$ analytic in $\T^n_{\sigma_\nu}$.\\
 \indent We set  $\Phi^{\nu}=\Phi^{\nu-1}\circ  \Phi_{\nu}:\
\T^n_{\sigma_{\nu+1}}\times\R^n\times Y_s\rightarrow \T^n_{\sigma_{0}}\times\R^n\times Y_{s}$ for $\omega\in D_\nu$ with $\Phi^0=id$ and
$\Phi_j(\theta,y,\xi,\eta)=X_{f_j}^1(\theta,y,\xi,\eta)=(\theta, \tilde{y}, e^{-\mathrm{i}F^T_{j}(\theta)}\xi,e^{\mathrm{i}F_{j}(\theta)}\eta)$ which is generated by the time 1 map of Hamiltonian function $f_j=\la\xi,F_j(\theta)\eta \ra$, $j=1,\cdots,\nu.$
 Under $\Phi^j$, we suppose that
$$H^{(j)}\circ \Phi^{j}=(h+q_j)\circ \Phi^{j}=h_{j}+p_{j},\ j=1,\cdots,\nu,$$
where
$h_{j}(\theta,y,\xi,\eta,\omega)=\la\omega,y\ra+\la  \xi,N_j(\omega)\eta\ra,$
$ p_j(\theta,y,\xi,\eta,\omega)=\la \xi,{P}_{j}(\theta,\omega)\eta\ra$
with $(\theta,y,\xi,\eta,\omega)\in\T^n_{\sigma_{j+1}}\times\R^n\times Y_s\times D_{j}$,
 and for $j=1,\cdots,\nu,$  the following estimates hold
\begin{eqnarray}\label{homoP01}
[P_j]_{s,\alpha}^{D_{j},\sigma_{j+1}}&\leq&  \frac{1}{2}\epsilon_{j};\\ \label{homoF01}
[F_j]_{s,\alpha+}^{D_{j},\sigma_{j+1}}&\leq&  c(n,d,s,\alpha)  \epsilon_{j-1}^{\frac{13}{24}}
;\\ \label{homoN01}
[{N}_j-N_{j-1}]_{s,\alpha}^{D_{j-1},\sigma_j}&\leq&  \epsilon_{j-1};\\
 \|\Phi_{j}-id\|^*_{\mathfrak{L}(Y_{s'},Y_{s'+2\alpha})}&\leq&  c(n,d,s,\alpha)  \epsilon_{j-1}^{\frac{13}{24}},\  0\leq s'\leq s;\\
{\rm Meas}(D_{j-1}\setminus D_j)&\leq& cK_{j-1}^{\gamma_1}\kappa_{j-1}^{\gamma_2}.
\end{eqnarray}
In the $(\nu+1)$th step  we  consider
$$H^{(\nu+1)}(\theta,y,\xi,\eta,\omega)=h+q_{\nu+1}=H^{(\nu )}+(q_{\nu+1}-q_{\nu })$$ with $ (\theta,y,\xi,\eta,\omega)\in
\T^n_{\sigma_{\nu+1}}\times\R^n\times Y_s\times  D_{\nu}.$
By $\Phi^\nu$  we have
\begin{eqnarray}\label{Hnu+1}
H^{(\nu+1)}\circ \Phi^{\nu}
 =H^{(\nu )}\circ \Phi^{\nu}+(q_{\nu+1}-q_{\nu })\circ \Phi^{\nu}:=h_{\nu}+\tilde{p}_{\nu},
\end{eqnarray}
where $\tilde{p}_{\nu}= p_{\nu}+(q_{\nu+1}-q_{\nu})\circ \Phi^{\nu}:=\la {\xi,\widetilde{P}}_{\nu}(\theta)\eta\ra$.\\
\indent We make some assumptions on parameters during the KAM iteration. For any $\nu=0,1,2,\cdots,$ the following assumptions hold:\\
\begin{eqnarray*}
&&{\rm B2}.\qquad
\sigma_{\nu+1}\leq \frac{1}{2}\sigma_\nu.
\label{canshu01} \\
&&{\rm B3}.\qquad
\sigma_{\nu}=\epsilon_{\nu-1}^{t_1}, \kappa_{\nu}=\epsilon_{\nu}^{t_2},\
K_{\nu-1}^{d+1}\leq\epsilon_{\nu-1}^{-\frac{1}{8}}\ {\rm with} \ nt_1\leq\frac{1}{6},\ (2+d/\alpha)t_2\leq\frac{1}{6}.
\label{canshu02}\\
&&{\rm B4}.\qquad \epsilon_{\nu+1}=\epsilon_{\nu}^{\frac{3}{2}}.
\label{canshu03}\\
&&{\rm B5}.\qquad \beta\geq\frac{3}{2t_1}.\label{canshu04}\\
&&{\rm B6}.\qquad e^{-\frac{1}{4}K_\nu\sigma_{\nu+1}}\leq\epsilon_\nu.\label{canshu05}
\end{eqnarray*}
\indent The explicit expressions of these parameters are given in Subsection {\ref{iteration}}.
Under these assumptions, we have
\begin{Lemma}\label{psigmanuandnuminus1}
Under Assumptions {\rm B2 - B5}, if $c(n,\beta,d,s)(1+ c(\alpha,s)\epsilon_{0}^{1/2})^2\varepsilon\leq\frac12$, then $\widetilde{P}_{\nu}$ is Hermitian  when $\theta\in\T^n$ and
$[\widetilde{P}_{\nu}]_{s,\alpha}^{D_\nu,\sigma_{\nu+1}}\leq\epsilon_{\nu}$.
\end{Lemma}
We need the following two preparation lemmas for Lemma \ref{psigmanuandnuminus1}.  \\
\indent Denote  $B_\nu=e^{\mathrm{i}F_1}\cdots e^{\mathrm{i}F_\nu}$, where  $B_\nu$ is defined on $(\theta,\omega)\in \T^n_{\sigma_{\nu+1}}\times  D_{\nu}$. Note that $F_\nu, \nu=1,\cdots,$ are Hermitian matrices, thus
\begin{equation}\label{inverseM}B^{-1}_\nu(\theta)=\overline{B}^T_\nu(\theta), \ \theta\in\T^n.
 \end{equation}
\begin{Lemma}\label{L4.6}
If $ 4\epsilon_{0}^{1/24}\leq 1$, then $B_\nu-Id,\ B^{-1}_\nu-Id
\in \mathcal{M}_{s,\alpha}^+{(D_{\nu},\sigma_{\nu+1})}$ and $$[B_\nu-Id]_{s,\alpha+}^{D_\nu,\sigma_{\nu+1}},\ [B^{-1}_\nu-Id]_{s,\alpha+}^{D_\nu,\sigma_{\nu+1}}  \leq  \epsilon_{0}^{\frac12}$$
under Assumptions {\rm B2-B4}. Moreover,  for $0\leq \nu_1<\nu$,
\begin{eqnarray*}
[B_{\nu_1}-B_{\nu}]_{s,\alpha+}^{D_\nu,\sigma_{\nu+1}},\ [B^{-1}_{\nu_1}-B^{-1}_{\nu}]_{s,\alpha+}^{D_\nu,\sigma_{\nu+1}}\leq
  \epsilon_{\nu_1}^{\frac12}.
\end{eqnarray*}
\end{Lemma}
\begin{proof} We only give the estimates on $[B_\nu-Id]_{s,\alpha+}^{D_\nu,\sigma_{\nu+1}}$ and $[B_{\nu_1}-B_{\nu}]_{s,\alpha+}^{D_\nu,\sigma_{\nu+1}}$.
From Lemma \ref{daishu} and (\ref{homoF01}), we have
\begin{eqnarray*}
 [e^{\mathrm{i}F_j}-Id]_{s,\alpha+}^{D_\nu,\sigma_{\nu+1}}\leq [F_j]_{s,\alpha+}^{D,\sigma}e^{c(\alpha,s )[F_j]_{s,\alpha+}^{D,\sigma}}\leq 2\epsilon_{j-1}^{\frac{13}{24}},\ j=1,\cdots,\nu,
\end{eqnarray*}
where  $e^{c(\alpha,s )\epsilon_{j-1}^{\frac{13}{24}}}\leq 2$. It follows that
\begin{eqnarray*}
 &&[e^{\mathrm{i}F_1}e^{\mathrm{i}F_2}-Id]_{s,\alpha+}^{D_\nu,\sigma_{\nu+1}}\\
 &\leq& [e^{\mathrm{i}F_1}-Id]_{s,\alpha+}^{D_\nu,\sigma_{\nu+1}}+[e^{\mathrm{i}F_2}-Id]_{s,\alpha+}^{D_\nu,\sigma_{\nu+1}}+ [(e^{\mathrm{i}F_1}-Id)(e^{\mathrm{i}F_2}-Id)]_{s,\alpha+}^{D_\nu,\sigma_{\nu+1}}\\
&\leq& 2\epsilon_{0}^{\frac{13}{24}} +2\epsilon_{1}^{\frac{13}{24}}+4c(\alpha,s)\epsilon_{0}^{\frac{13}{24}} \epsilon_{1}^{\frac{13}{24}} \leq 3\epsilon_{0}^{\frac{13}{24}} +3\epsilon_{1}^{\frac{13}{24}}.
\end{eqnarray*}
By induction, we obtain
\begin{eqnarray}\label{mapit}
 [e^{\mathrm{i}F_1}\cdots e^{\mathrm{i}F_{\nu}}-Id]_{s,\alpha+}^{D_\nu,\sigma_{\nu+1}}\leq 3\epsilon_{0}^{\frac{13}{24}} +3\epsilon_{1}^{\frac{13}{24}}+\cdots+3\epsilon_{\nu-1}^{\frac{13}{24}} \leq   \epsilon_{0}^{\frac{1 }{2 }}
 \end{eqnarray}
 by   Assumption  B4.\\
\indent Following a similar discussion above, we also have  for $0\leq \nu_1<\nu$,
\begin{eqnarray}\label{mapit01}
  [e^{\mathrm{i}F_{\nu_1+1}}\cdots e^{\mathrm{i}F_{\nu}}-Id]_{s,\alpha+}^{D_\nu,\sigma_{\nu+1}}\leq  C\epsilon_{\nu_1}^{\frac{13}{24}}.
\end{eqnarray}
\indent Note that for $0\leq \nu_1<\nu$, $B_{\nu_1}-B_{\nu}=B_{\nu_1}(Id-e^{\mathrm{i}F_{\nu_1+1}}\cdots e^{\mathrm{i}F_{\nu}})$.
Thus
\begin{eqnarray*}
&&[B_{\nu_1}-B_{\nu_2}]_{s,\alpha+}^{D_\nu,\sigma_{\nu+1}}\\
&\leq& [(B_{\nu_1}-Id)(Id-e^{\mathrm{i}F_{\nu_1+1}}\cdots e^{\mathrm{i}F_{\nu_2}})]_{s,\alpha+}^{D_\nu,\sigma_{\nu+1}}+[Id-e^{\mathrm{i}F_{\nu_1+1}}\cdots e^{\mathrm{i}F_{\nu_2}}]_{s,\alpha+}^{D_\nu,\sigma_{\nu+1}}\\
&\leq&  c(\alpha,s,n)\epsilon_0^{\frac12}\epsilon_{\nu_1}^{\frac{13}{24}}+ \epsilon_{\nu_1}^{\frac{13}{24}}
\ \leq\   \epsilon_{\nu_1}^{\frac12}.
\end{eqnarray*}
 by (\ref{exponorm01}), (\ref{exponorm02}), (\ref{mapit}) and (\ref{mapit01}) .
\end{proof}

\begin{Lemma}\label{L4.5}
If $c(n,\beta,d,s)(1+ c(\alpha,s)\epsilon_{0}^{1/2})^2\varepsilon\leq\frac12$, then $[B^{-1}_\nu(Q_{{\nu+1}}-Q_{{\nu}})B_\nu]_{s,\alpha}^{D_\nu,\sigma_{\nu+1}}\leq\frac{1}{2}\epsilon_\nu$
under Assumptions {\rm B2-B5}.
\begin{proof}
Firstly, we consider $[(Q_{{\nu+1}}-Q_{{\nu}})B_\nu]_{s,\alpha}^{D_\nu,\sigma_{\nu+1}}$ where $$[ Q_{{\nu+1}}-Q_{{\nu}} ]_{s,\alpha}^{D_\nu,\sigma_{\nu+1}}=\varepsilon[  P^{(\nu+1)}-P^{(\nu)} ]_{s,\alpha}^{D_\nu,\sigma_{\nu+1}}\leq c(n,\beta,d,s)\varepsilon \sigma_\nu^{\beta}$$ by Lemma \ref{PP}. Then from Lemma \ref{daishu01} and Lemma \ref{L4.6}, it holds that
\begin{eqnarray*}
 &&[(Q_{{\nu+1}}-Q_{{\nu}})B_\nu]_{s,\alpha}^{D_\nu,\sigma_{\nu+1}}\\
 &\leq& [(Q_{{\nu+1}}-Q_{{\nu}})(B_\nu-Id)]_{s,\alpha}^{D_\nu,\sigma_{\nu+1}}+[Q_{{\nu+1}}-Q_{{\nu}}]_{s,\alpha}^{D_\nu,\sigma_{\nu+1}}\\
&\leq& c(n,\beta,d,s)(1+ c(\alpha,s)\epsilon_{0}^{1/2})\varepsilon \sigma_\nu^{\beta},
\end{eqnarray*}
and
\begin{eqnarray*}
&&[B^{-1}_\nu(Q_{{\nu+1}}-Q_{{\nu}})B_\nu]_{s,\alpha}^{D_\nu,\sigma_{\nu+1}}\\
 &\leq& (1+c(\alpha,s)[B^{-1}_\nu-Id]_{s,\alpha+}^{D_\nu,\sigma_{\nu+1}}) [(Q_{{\nu+1}}-Q_{{\nu}})B_\nu]_{s,\alpha}^{D_\nu,\sigma_{\nu+1}}\\
&\leq& c(n,\beta,d,s)(1+ c(\alpha,s)\epsilon_{0}^{1/2})^2\varepsilon \sigma_\nu^{\beta} \leq\frac{1}{2}\epsilon_\nu
\end{eqnarray*}
under Assumptions B2-B5.
\end{proof}
\end{Lemma}

Note that $\widetilde{P}_{\nu}=P_{\nu}+B^{-1}_\nu(Q_{{\nu+1}}-Q_{{\nu}})B_\nu$, then it is easy to check that $\widetilde{P}_{\nu}$ is Hermitian. Lemma \ref{psigmanuandnuminus1} is obtained immediately from (\ref{homoP01}) and Lemma \ref{L4.5}.   $\Box$

 Go back to the Hamiltonian (\ref{Hnu+1}),  we write $\widetilde{p}_{\nu}=\Gamma\widetilde{p}_{\nu}+r_{\nu}$ and $\widetilde{P}_{\nu}=\Gamma\widetilde{P}_{\nu}+R_{\nu}$ respectively,
where
\begin{eqnarray*}
  \Gamma\widetilde{P}_{\nu}(\theta,\omega):= \sum\limits_{|k|\leq K_{\nu}} \widetilde{P}^k_{\nu}(\omega)e^{{{\rm i}k\cdot\theta}},\qquad
   R_{\nu}(\theta,\omega):= \sum\limits_{|k|> K_{\nu}} \widetilde{P}^k_{\nu}(\omega)e^{{{\rm i}k\cdot\theta}},
\end{eqnarray*}
and define $
\widetilde{N}_{\nu}(\omega)\in \mathcal{NF}
$ satisfying
\begin{eqnarray}\label{ntilde}
\left(\widetilde{N}_{\nu}(\omega)\right)_{[a]}:= \left(\widetilde{P}^0_{\nu}(\omega)\right)_{[a]}.
\end{eqnarray}
In the following we will  use $\Phi_{\nu+1}=X_{f_{\nu+1}}^1$ with $f_{\nu+1}=\la \xi,F_{\nu+1}(\theta)\eta\ra$ to put $\Gamma\widetilde{p}_{\nu}$ into normal form.
Assume $f_{\nu+1}=\la \xi,F_{\nu+1}(\theta)\eta\ra$ satisfying
\begin{equation}\label{homoeq03}
\Gamma \tilde{p}_{\nu}-\la\xi,\widetilde{N}_{\nu}\eta\ra+\{h_{\nu},\ f_{\nu+1}\}=0,
\end{equation}
then under $\Phi_{\nu+1}(\xi,\eta)=X_{f_{\nu+1}}^1(\xi,\eta)=(e^{-\mathrm{i}F^T_{\nu+1}}\xi,e^{\mathrm{i}F_{\nu+1}}\eta)$, we have
$
H^{(\nu+1)}\circ \Phi^{\nu}\circ \Phi_{\nu+1}
= h_{\nu+1}+p_{\nu+1},
$
where
$h_{\nu+1}=\la\omega,y\ra+\la \xi,N_{\nu+1}\eta\ra\ {\rm with}\ N_{\nu+1}=N_{\nu}+\widetilde{N}_{\nu}$,
and $p_{\nu+1}=\la \xi,{P}_{\nu+1}(\theta,\omega)\eta\ra$
with
\begin{eqnarray}\label{newP}
P_{\nu+1}&=&-\mathrm{i}\int_0^1e^{-\mathrm{i}tF_{\nu+1} }[t\Gamma \widetilde{P}_{\nu}+(1-t)\widetilde{N}_\nu,F_{\nu+1}]e^{\mathrm{i}tF_{\nu+1} }dt+
e^{-\mathrm{i}F_{\nu+1} }R_{\nu}e^{\mathrm{i}F_{\nu+1} }.
\end{eqnarray}
  Therefore,  we obtain
$$H^{(\nu+1)}\circ \Phi^{\nu+1}=h_{\nu+1}+p_{\nu+1}=\la\omega,y\ra+\la\xi, N_{\nu+1}(\omega)\eta\ra+\la \xi,{P}_{\nu+1}(\theta,\omega)\eta\ra.$$
In the following we will give the explicit  estimates on  $F_{\nu+1}$, $N_{\nu+1}$,
$P_{\nu+1}$ and $\Phi_{\nu+1}-id$.\\
First of all, (\ref{homoeq03}) is
equivalent to
\begin{eqnarray}\label{homoequations}
\omega\cdot\nabla_\theta F_{\nu+1}(\theta)-i[N_{\nu},F_{\nu+1}(\theta)]+\Gamma\widetilde{P}_{\nu}(\theta)=\widetilde{N}_{\nu}+R_\nu(\theta).
\end{eqnarray}
From Proposition \ref{pro4.1} and Assumption {\rm B4}, we have from (\ref{homoequations})
\begin{eqnarray}\label{homon02}
[N_\nu-N_0]_{s,\alpha}^{D_\nu}\leq\sum_{j=1}^\nu[N_j-N_{j-1}]_{s,\alpha}^{D_\nu}\leq\epsilon_0+\cdots+\epsilon_{\nu-1}\leq 2\epsilon_0\leq\frac{c_0}{4}.
\end{eqnarray}
It follows that
\begin{eqnarray*}\label{homon02-1}
\|\partial_\omega^j(N_\nu-N_0)_{[a]}\|w_a^{2\alpha}\leq 2\epsilon_0\ j=0,1.
\end{eqnarray*}
From Proposition \ref{pro4.1}, together with (\ref{homon02}), Lemma \ref{psigmanuandnuminus1}  and Assumptions {\rm B2-B6}, if
$\kappa_\nu\leq\frac{c_0}{2}$ and $K_\nu\geq1$ then there exists a subset $D_{\nu+1}\subset D_\nu$ with ${\rm Meas}(D_\nu\setminus D_{\nu+1})\leq c
K_\nu^{\gamma_1}\kappa_\nu^{\gamma_2}$, and there exist $\widetilde{N}_\nu\in \mathcal{M}_{s,\alpha}(D_{\nu+1})\cap\mathcal{NF},$ $F_{\nu+1}\in
\mathcal{M}_{s,\alpha}^+{(D_{\nu+1},\sigma_{\nu+2})}$, $R_\nu\in \mathcal{M}_{\alpha}{(D_{\nu+1},\sigma_{\nu+2})}$, $\mathcal{C}^1$ in $\omega$ and analytic in $\theta$, such that
(\ref{homoequations}) holds for all  $(\theta,\omega)\in \T^n_{\sigma_{\nu+2}}\times D_{\nu+1}$ and
\begin{eqnarray}
 \label{homoFn}
[F_{\nu+1}]_{s,\alpha+}^{D_{\nu+1},\sigma_{\nu+2}}&\leq&  \frac{c(n,d,s,\alpha)\epsilon_\nu K_\nu^{d+1}}{\kappa_\nu^{2+d/\alpha}(\sigma_{\nu+1}-\sigma_{\nu+2})^{n}}\leq c(n,d,s,\alpha)\epsilon_\nu^{\frac{13}{24}}
;\\ \label{homoNn}
[\widetilde{N}_\nu]_{s,\alpha}^{D_{\nu+1}}&\leq&  [\widetilde{P}_{\nu}]_{s,\alpha}^{D_{\nu},\sigma_{\nu+1}}\leq\epsilon_\nu.\\
\label{homoRn}[R_\nu]_{s,\alpha}^{D_{\nu+1},\sigma_{\nu+2}}&\leq&
\frac{c(n,d,s,\alpha)\epsilon_\nu K_{\nu}^{1+d/2}e^{-(\sigma_{\nu+1}-\sigma_{\nu+2})K_{\nu}/2}}{\kappa_\nu^{1+d/2\alpha}(\sigma_{\nu+1}-\sigma_{\nu+2})^{n}}\leq \epsilon_{\nu+1}.
\end{eqnarray}
$\widetilde{N}_\nu,$ $F_{\nu+1}(\theta) $  and $R_\nu(\theta) $ are Hermitian when $\theta\in\T^n$.
Moreover, from Lemmas \ref{daishu01}, \ref{daishu} and (\ref{homoFn}),
\begin{Lemma}\label{map}
If $4c(\alpha,s)\varepsilon^{\frac{1 }{24}}\leq 1$, then the symplectic map $\Phi_{\nu+1}(\theta,\omega)$ defined in $\T^n_{\sigma_{\nu+2}}\times D_{\nu+1}$  satisfies
\begin{eqnarray*}
\|\Phi_{\nu+1}-id\|^*_{\mathfrak{L}(Y_{s'},Y_{s'+2\alpha})}\leq
4c(\alpha,s)\epsilon_\nu^{\frac{13}{24}}\leq \epsilon_\nu^{\frac{1 }{2 }}.
\end{eqnarray*}
for all $0\leq s'\leq s$.
\begin{proof}
Note that $\Phi_{\nu+1}(\xi,\eta)=X_{f_{\nu+1}}^1(\xi,\eta)=(e^{-\mathrm{i}F^T_{\nu+1}}\xi,e^{\mathrm{i}F_{\nu+1}}\eta)$, where
\begin{eqnarray*}
 \|\partial^k_\omega (e^{\mathrm{i}F_{\nu+1}}-Id)\eta\|_{s'+2\alpha} &\leq&  c(\alpha,s)[e^{\mathrm{i}F_{\nu+1}}-Id]_{s,\alpha+}\|\eta\|_{s'}\\
 &\leq& c(\alpha,s)[F_{\nu+1}]_{s,\alpha+}^{D_{\nu+1},\sigma_{\nu+2}}e^{c(\alpha,s )[F_{\nu+1}]_{s,\alpha+}^{D_{\nu+1},\sigma_{\nu+2}}} \|\eta\|_{s'}\\
 &\leq& 2c(\alpha,s)\epsilon_\nu^{\frac{13}{24}}\|\eta\|_{s'}
 \end{eqnarray*}
for $k=0,1$ and $0\leq s'\leq s$ by (\ref{homoFn}), Lemmas \ref{daishu01} and \ref{daishu}. Similarly,
\begin{eqnarray*}
 \|\partial^k_\omega (e^{-\mathrm{i}F^T_{\nu+1}}-Id)\xi\|_{s'+2\alpha} \leq 2c(\alpha,s)\epsilon_\nu^{\frac{13}{24}}\|\xi\|_{s'}
 \end{eqnarray*}
for $k=0,1$.
Therefore, we have
$
\|\Phi_{\nu+1}-id\|^*_{\mathfrak{L}(Y_{s'},Y_{s'+2\alpha})}
\leq c(\alpha,s,n,d)\epsilon_\nu^{\frac{13}{24}}. \label{homoPhin}
$
\end{proof}
\end{Lemma}

\noindent{\it Estimates on the new error term.}
Recall
$
p_{\nu+1}
=\la \xi,P_{\nu+1}(\theta)\eta\ra
$
with (\ref{newP}),
we have
\begin{Lemma}\label{pnu+1}
Under Assumptions {\rm B2} - {\rm B6}, if $\varepsilon\leq c^{-1}(n,d,s,\alpha)\ll1$ then $P_{\nu+1}(\theta)$ is Hermitian when $\theta\in\T^n$ and
$
[P_{\nu+1}]_{s,\alpha}^{D_{\nu+1},\sigma_{\nu+2}}\leq \frac12\epsilon_{\nu+1}.
$
\end{Lemma}
\begin{proof} It is easy to check that $P_{\nu+1}(\theta)$ is Hermitian from (\ref{newP}).
The estimate is divided into two parts.\\
1) From (\ref{exponorm01}), (\ref{canshu05}),  (\ref{homoFn}) and (\ref{homoRn}),
\begin{eqnarray*}
[e^{-\mathrm{i}F_{\nu+1}}R_{\nu}e^{\mathrm{i}F_{\nu+1}}]_{s,\alpha}^{D_{\nu+1},\sigma_{\nu+2}}
&\leq & e^{4c(\alpha)[F_{\nu+1}]_{s,\alpha+}^{D_{\nu+1},\sigma_{\nu+2}}}[R_{\nu}]_{s,\alpha}^{D_{\nu+1},\sigma_{\nu+2}} \\
&\leq & e^{4}[R_{\nu}]_{s,\alpha}^{D_{\nu+1},\sigma_{\nu+2}}\leq e^{4}\epsilon_\nu^{\frac{1 }{8}}\epsilon_\nu^{\frac{3 }{2}}
\leq\frac{1}{4}\epsilon_{\nu}^{\frac32}.
\end{eqnarray*}
2) For $0\leq t\leq1$, note that $\Gamma \widetilde{P}_{\nu}=\widetilde{P}_{\nu}-R_{\nu},$
\begin{eqnarray*}
&&[e^{-\mathrm{i}tF_{\nu+1}}[t\Gamma \widetilde{P}_{\nu}+(1-t)\widetilde{N}_\nu,F_{\nu+1}]e^{\mathrm{i}tF_{\nu+1}}]_{s,\alpha}^{D_{\nu+1},\sigma_{\nu+2}}\\
&\leq
&e^{4c(\alpha)[F_{\nu+1}]_{s,\alpha+}^{D_{\nu+1},\sigma_{\nu+2}}}[t\Gamma \widetilde{P}_{\nu}+(1-t)\widetilde{N}_\nu,F_{\nu+1}]_{s,\alpha}^{D_{\nu+1},\sigma_{\nu+2}}\\
&\leq
&2c(\alpha)e^{4c(\alpha)[F_{\nu+1}]_{s,\alpha+}^{D_{\nu+1},\sigma_{\nu+2}}}[F_{\nu+1}]^{D_{\nu+1},\sigma_{\nu+2}}_{s,\alpha+}
([\Gamma \widetilde{P}_{\nu}]^{D_{\nu+1},\sigma_{\nu+2}}_{s,\alpha}+[\widetilde{N}_\nu]^{D_{\nu+1}}_{s,\alpha})\\
&\leq &4c(\alpha)e^{4}\epsilon_{\nu}^{\frac{1}{24}}\epsilon_{\nu}^{\frac32}\leq \frac{1}{4}\epsilon_\nu^{\frac32}\end{eqnarray*}
 by  Lemma \ref{daishu}, (\ref{exponorm01}), (\ref{canshu05}), (\ref{homoFn}), (\ref{homoNn}), (\ref{homoRn}).
\end{proof}

\subsection{Iteration Lemma}\label{iteration}

To iterate the KAM steps infinitely often we choose sequences for the pertinent parameters.  Set $\epsilon_0=2\varepsilon c(n,\beta,d,s)$, $\sigma_0=1$ and suppose
\begin{eqnarray}
\epsilon_{\nu+1}=\epsilon_{\nu}^{\frac{3}{2}},\sigma_{\nu+1}=\epsilon_{\nu}^{\frac{3}{2\beta}}, \kappa_{\nu}=\epsilon_{\nu}^{\frac{1}{6(2+d/\alpha)}},
K_{\nu}=8|\ln \epsilon_{\nu}|\epsilon_{\nu}^{-\frac{3}{2\beta}}\label{canshudingyi}
\end{eqnarray}
 for $\nu\geq 0$. From a straightforward computation in Subsection \ref{KAMstep} we have
\begin{Lemma}[Iterative Lemma]\label{itelm}
Let
$0<\delta<\frac{\gamma_2}{24}$ and
\begin{eqnarray}
\beta>\beta_*=\max\{9(2+d/\alpha)\frac{\gamma_1}{\gamma_2-24\delta},\ 9n,\ 12(d+1)\}\label{betaxinzhi}
\end{eqnarray}
with
 $\gamma_1=\max\{d+n+2,\alpha_1\},\ \gamma_2=\frac{\alpha\alpha_2}{4+d+2\alpha\alpha_2}$.
If $0<\varepsilon\leq\varepsilon _*(n,d,s,\delta)\ll 1$
then all the iteration series $\epsilon_{\nu},  \sigma_{\nu}, \kappa_{\nu}$ and $K_{\nu}$
satisfy Assumptions $\mathrm{B1} - \mathrm{B6}$,
therefore we have the followings:\\
 Suppose that in $\T^n_{\sigma_{\nu+1}}\times\R^n\times Y_s\times D_\nu$,
$$H^{(\nu)}\circ \Phi^{\nu}(\theta,y,\xi,\eta,\omega)=(h+q_\nu)\circ \Phi^{\nu}=h_{\nu}+p_{\nu},$$
where
$\Phi^{\nu}=\Phi_1\circ \Phi_2\circ\cdots  \Phi_{\nu}$ and
$\Phi_{j}=X_{f_j}^1: Y_{s'}\rightarrow Y_{s'}$ for all $0\leq s'\leq s$, $\omega \in D_{j}$ and $\theta\in \T_{\sigma_{j+1}}$ and
satisfies
$$\|\Phi_{j}-id\|^*_{\mathfrak{L}(Y_{s'},Y_{s'+2\alpha})}\leq \epsilon_{j-1}^{\frac{1}{2}}, j=1, \cdots, \nu,$$
and $h_{\nu}=\la \omega,y\ra+\la\xi,N_\nu \eta\ra$ in normal form and $p_{\nu}=\la \xi,{P}_{\nu}\eta\ra$ where
$N_{\nu}-N_0\in \mathcal{M}_{s,\alpha}(D_{\nu})$, $P_{\nu}\in \mathcal{M}_{s,\alpha}(D_{\nu},\sigma_{\nu+1})$  and
the following estimates hold:
\begin{eqnarray*}
\nonumber
{\rm Meas}(D_{j-1}\setminus D_{j}) &\leq&  cK_{j-1}^{\gamma_1}\kappa_{j-1}^{\gamma_2};\\ \nonumber
[F_{j}]_{s,\alpha+}^{D_{j}, \sigma_{j+1}}&\leq&  c(n,d,s,\alpha)  \epsilon_{j-1}^{\frac{13}{24}};\\ \nonumber
[N_j-N_{j-1}]_{s,\alpha}^{D_{j-1},\sigma_j}&\leq&
  \epsilon_{j-1};\\ \nonumber
[P_{j}]_{s,\alpha}^{D_{j}, \sigma_{j+1}}&\leq&    \frac12\epsilon_{j}.
\end{eqnarray*}
Then  there exist $ D_{\nu+1}\subset D_\nu$ and a mapping
$\Phi_{\nu+1}=X_{f_{\nu+1}}^1: Y_{s'}\rightarrow Y_{s'}$ for all $0\leq s'\leq s$,   $\omega\in D_{\nu+1},\ \theta\in \T_{\sigma_{\nu+2}}$,
and
$$H^{(\nu+1)}\circ \Phi^{\nu}\circ \Phi_{\nu+1}=(h+q_{\nu+1})\circ \Phi^{\nu+1}=h_{\nu+1}+p_{\nu+1}$$
and $h_{\nu+1}=\la \omega,y\ra+\la\xi,N_{\nu+1} \eta\ra$ in normal form and $p_{\nu+1}=\la \xi,{P}_{\nu+1}\eta\ra$where
$N_{\nu+1}-N_0\in \mathcal{M}_{s,\alpha}(D_{\nu+1})$, $P_{\nu+1}\in \mathcal{M}_{s,\alpha}(D_{\nu+1},\sigma_{\nu+2})$  and
the same estimates hold for $j\leq\nu+1$.\\
\indent Moreover, since $N_{\nu}, P_{\nu}(\theta)$ are Hermitian when $\theta\in\T^n$, so are $N_{\nu+1}, P_{\nu+1}(\theta), F_{\nu+1}(\theta)$. \quad\quad$\Box$
\end{Lemma}

\subsection{Transition to the limit and the proof of Theorem \ref{MainTheorem}}

Set $ D_\varepsilon=\bigcap_{\nu=0}^\infty  D_\nu$. From (\ref{canshudingyi}), (\ref{betaxinzhi}) and Lemma \ref{itelm} we have
$${\rm Meas}(D_0\setminus D_\varepsilon)\leq\sum_{\nu=0}^\infty cK_{\nu}^{\gamma_1}\kappa_{\nu}^{\gamma_2}\leq c(n,d,s,\beta,\delta) \varepsilon^{\frac{3\delta}{2+d/\alpha}},$$
 if $0<\varepsilon\leq \varepsilon_*(n,d,s,\beta,\delta)$.   \\
 \indent In the following we will show that
 \begin{Lemma}\label{L4.12}
 For $(\omega, \theta)\in D_{\varepsilon}\times \T^n$,
 $\{\Phi^{\nu}-id\}_\nu $ is a Cauchy sequence in $\mathfrak{L}(Y_{s'}, Y_{s'+2\alpha})$ for all $0\leq s'\leq s$.
 \end{Lemma}
To prove this lemma we need to show that
\begin{Lemma}\label{zhongjianguji}
For   $0\leq s'\leq s$, $0\leq {\nu}_1< \nu_2$ and $(\omega,\theta)\in D_{\varepsilon}\times \T^n$,
$$\|\Phi_{\nu_1+1}\circ \Phi_{\nu_1+2}\circ \cdots \circ \Phi_{\nu_2}-id \|^*_{\mathfrak{L}(Y_{s'}, Y_{s'+2\alpha})}\leq  C\epsilon_{\nu_1}^{\frac12}. $$
\end{Lemma}
\begin{proof}
Similar to the proof of Lemma \ref{map}, from Lemma \ref{daishu01} iv) and (\ref{mapit01}),  for   all $0\leq s'\leq s$, $(\theta,\omega)\in\T^n\times { D_\varepsilon} $ it holds that
\begin{eqnarray*}
 &&\|\Phi_{\nu_1+1}\circ \Phi_{\nu_1+2}\circ \cdots \circ \Phi_{\nu_2}-id \|^*_{\mathfrak{L}(Y_{s'}, Y_{s'+2\alpha})}\\
 &\leq& c(\alpha,s)\max\left\{[e^{-\mathrm{i}F^T_{\nu_1+1}}\cdots e^{-\mathrm{i}F^T_{\nu_2}}-Id]_{s,\alpha+}^{ D_\varepsilon},\ [e^{\mathrm{i}F_{\nu_1+1}}\cdots e^{\mathrm{i}F_{\nu_2}}-Id]_{s,\alpha+}^{ D_\varepsilon}\right\}\ \leq\  C \epsilon_{\nu_1}^{1/2}.
\end{eqnarray*}
\end{proof}
From Lemma \ref{zhongjianguji}, let $\nu_1=0$ we have
\begin{Corollary}\label{coro4.1}
For   all $0\leq s'\leq s$, $(\omega,\theta)\in D_{\varepsilon}\times \T^n$,
$$\|\Phi^{\nu}-id\|^*_{\mathfrak{L}(Y_{s'}, Y_{s'+2\alpha})}\leq  C\epsilon_0^{\frac12}.$$
\end{Corollary}

\noindent\emph{Proof of Lemma \ref{L4.12}}.\ \  Recall that
\begin{eqnarray*}
 \Phi^{\nu_1}-\Phi^{\nu_2}&=&\Phi^{\nu_1}\circ(\Phi_{\nu_1+1}\circ \Phi_{\nu_1+2}\circ \cdots \circ \Phi_{\nu_2}-id)\\
 &=&(\Phi^{\nu_1}-id)\circ(\Phi_{\nu_1+1}\circ \Phi_{\nu_1+2}\circ \cdots \circ \Phi_{\nu_2}-id)+\Phi_{\nu_1+1}\circ \Phi_{\nu_1+2}\circ \cdots \circ \Phi_{\nu_2}-id,
 \end{eqnarray*}
 Then by Lemmas \ref{zhongjianguji} and Corollary \ref{coro4.1} we have
$
\|\Phi^{\nu_1}-\Phi^{\nu_2}\|^*_{\mathfrak{L}(Y_{s'}, Y_{s'+2\alpha})}\leq C\epsilon_{\nu_1}^{\frac12}.
$

We recall that $\Phi^{\nu}-id$ is a map from $\T^n_{\sigma_{\nu+1}}\times D_{\varepsilon}$ to an operator in $\mathfrak{L}(Y_{s'}, Y_{s'+2\alpha})$. From Lemmas \ref{completed} and \ref{L4.12} we
denote
$\Phi_\omega^\infty-id =\lim\limits_{\nu\rightarrow \infty}(\Phi^{\nu}-id)$ which is a map from $\T^n\times D_{\varepsilon}$ to  $\mathfrak{L}(Y_{s'}, Y_{s'+2\alpha})$.
Furthermore,
\begin{Lemma}\label{phiwuqiong}
For $(\omega, \theta)\in D_{\varepsilon}\times \T^n$,  $\Phi_\omega^\infty-id \in \mathcal{C}^{\mu}(\T^n, \mathfrak{L}(Y_{s'}, Y_{s' +2\alpha}))$ for all $0\leq s'\leq s$, where
$\mu\leq \frac{2}{9}\beta$ and is not an integer,
\begin{eqnarray*}
\|\Phi_\omega^\infty-id\|_{\mathcal{C}^{\mu}(\T^n, \mathfrak{L}(Y_{s'}, Y_{s'+2\alpha})) }&\leq&   \frac{c(n,\beta )}{\iota(1-\iota)}  \epsilon_0^{\frac{3}{2\beta}(\frac{2}{9}\beta-\mu)},\ 0<\iota:=\mu-[\mu]<1.
\end{eqnarray*}
\end{Lemma}
\begin{proof}
$$\|\Phi^{\nu}-\Phi^{\nu-1}\|_{\mathfrak{L}(Y_{s'},Y_{s'+2\alpha})}\leq C\epsilon_{\nu-1}^{\frac{1}{2}}\leq C\sigma_{\nu+1}^{\frac{2\beta}{9}}.$$
From Lemma \ref{smoothinginverse}, we deduce that $\Phi_\omega^\infty-id\in\mathcal{C}^{\mu}(\T^n,\mathfrak{L}(Y_{s'}, Y_{s'+2\alpha}))$ for every $\mu\leq \frac{2}{9}\beta$ which is not an integer and, \begin{eqnarray*}
\|\Phi_\omega^\infty-id\|_{\mathcal{C}^{\mu}(\T^n,\mathfrak{L}(Y_{s'}, Y_{s'+2\alpha}))}\leq   \frac{c(n,\beta )}{\iota(1-\iota)} \sigma_1^{\frac{2}{9}\beta-\mu}=   \frac{c(n,\beta )}{\iota(1-\iota)}  \epsilon_0^{\frac{3}{2\beta}(\frac{2}{9}\beta-\mu)},\ 0<\iota:=\mu-[\mu]<1.
\end{eqnarray*}
\end{proof}
\begin{Remark}
In fact we can prove that $\Phi_{\omega}^{\infty}$ is also $C^1$ smooth in $\omega$ in Whitney sense and satisfies a similar estimation as above.
\end{Remark}

 \begin{Lemma}\label{convergence}
  For any $y\in\R^n$,\ $(\xi,\eta)\in Y_{s'}$ with $1\leq s'\leq \max\{s,1\}$,
$$\mathcal{H}\circ\Phi_\omega^\infty (\xi,\eta)=h_\infty(\xi,\eta):=\la\omega,y\ra+\la\xi,N_\infty\eta\ra,$$
uniformly for $\omega\in D_\varepsilon$ and $\theta\in\T^n$, where
 $N_\infty(\omega)=\lim_{\nu\rightarrow\infty}N_{\nu}(\omega)$ with  $N_{\infty}(\omega)\in\mathcal{NF}$,
   uniformly on $ D_\varepsilon$,  $\mathcal{C}^1$ Whitney smooth and
$
[N_\infty-N_0]^{ D_\varepsilon}_{s,\alpha}\leq  2\epsilon_0.$
\end{Lemma}
 We need to prove a series of preparation lemmas.
 \begin{Lemma}\label{convergence01}
  For any $y\in\R^n$,\ $(\xi,\eta)\in Y_{\bar{s}}$ and $\bar{s}\geq 1$,
$$\lim_{\nu\rightarrow\infty}(h_\nu+p_\nu)=h_\infty=\la\omega,y\ra+\la\xi,N_\infty\eta\ra,$$
   uniformly on $ D_\varepsilon\times \T^n$  and
$
[N_\infty-N_0]^{ D_\varepsilon}_{s,\alpha}\leq  2\epsilon_0.$
\end{Lemma}
\begin{proof}
For $\nu_1<\nu_2$, from Lemma \ref{itelm},
$
[N_{\nu_1}-N_{\nu_2}]_{s,\alpha}^{D_\varepsilon}\leq 2\epsilon_{\nu_1}.
$
It follows that $N_\nu-N_0$ is a Cauchy series in the norm $[\cdot]_{s,\alpha}^{D_{\varepsilon, \T^n}}$ by (\ref{homon02}).  We denote $N_\infty=\lim N_\nu$ by Lemma \ref{completed}.
Clearly, $\nu\geq0$, $$[N_\infty-N_{\nu}]^{ D_\varepsilon}_{s,\alpha}\leq \sum_{k=\nu}^\infty[\widetilde{N}_{k}]^{ D_\varepsilon}_\alpha\leq 2\epsilon_\nu.$$
Thus $$|\la\xi,(N_\infty-N_\nu)\eta\ra|\leq\|\xi\|_{\bar{s}}\|(N_\infty-N_\nu)\eta\|_{-\bar{s}}\leq c(\alpha,s)[N_\infty-N_\nu]_{s,\alpha}^{D_\varepsilon}\|\xi\|_{\bar{s}}\|\eta\|_{\bar{s}}\leq 2c(\alpha,s)\epsilon_{\nu}\|\xi\|_{\bar{s}}\|\eta\|_{\bar{s}},$$
which means that
$\lim_\nu\la\xi,N_\nu\eta\ra=\la\xi,N_\infty\eta\ra.$
On the other hand,
\begin{equation*}\label{pes}|p_\nu|\leq c(\alpha,s)[P_\nu]_{s,\alpha}^{D_\nu,\sigma_{\nu+1}}\|\xi\|_{\bar{s}}\|\eta\|_{\bar{s}}\leq \frac{1}{2} c(\alpha,s)\epsilon_\nu \|\xi\|_{\bar{s}}\|\eta\|_{\bar{s}}\rightarrow 0,\ \nu\rightarrow\infty.
 \end{equation*}
 \end{proof}

 \begin{Lemma}\label{App02}
If $\epsilon_0 \ll 1,$ then for $0\leq s'\leq s$ and  $(\theta,\omega)\in \T^n \times  D_{\varepsilon}$ there exists $M_\omega(\theta)$  defined on $\mathfrak{L}(\ell^2_{s' },\ell^2_{s' })(0\leq s'\leq s)$ such that\\
i) $\|B_\nu-M_\omega\|_{\mathfrak{L}(\ell^2_{s'},\ell^2_{s'+2\alpha})},\ \|B_\nu^{-1}-M_\omega^{-1} \|_{\mathfrak{L}(\ell^2_{s'},\ell^2_{s'+2\alpha})}\leq c(\alpha,s,n)\epsilon_\nu^{\frac12};$ \\
ii) $\|M_\omega-Id\|_{\mathfrak{L}(\ell^2_{s'},\ell^2_{s'+2\alpha})},\ \|M_\omega^{-1} -Id\|_{\mathfrak{L}(\ell^2_{s'},\ell^2_{s'+2\alpha})}\leq c(\alpha,s,n)\epsilon_0^{\frac12}$.
\end{Lemma}
\begin{proof}
From Lemmas \ref{daishu01} and \ref{L4.6}, for $0\leq \nu_1<\nu_2$,
 $\|B_{\nu_1}-B_{\nu_2}\|_{\mathfrak{L}(\ell^2_{s'},\ell^2_{s'+2\alpha})}\leq c(\alpha,s,n) \epsilon_{\nu_1}^{\frac12}$.
 It means that $\{B_\nu-Id\}$ is a Cauchy sequence in $\mathfrak{L}(\ell^2_{s'},\ell^2_{s'+2\alpha})$ and its limit is denoted by $M_\omega-Id$, which satisfies
$\|B_\nu-M_\omega\|_{\mathfrak{L}(\ell^2_{s'},\ell^2_{s'+2\alpha})}\leq c(\alpha,s,n) \epsilon_\nu^{\frac12}$.
Recall that from Lemma \ref{L4.6},
$
[B_\nu-Id]_{s,\alpha+}^{\T^n,  D_{\varepsilon}}\leq \epsilon_0^{\frac12}.
$
It follows that
 $\|B_\nu-Id\|_{\mathfrak{L}(\ell^2_{s'},\ell^2_{s'+2\alpha})}\leq c(\alpha,s,n)\epsilon_0^{\frac12}$.
Set $\nu\rightarrow\infty,$ we obtain
$$\|M_\omega-Id\|_{\mathfrak{L}(\ell^2_{s'},\ell^2_{s'+2\alpha})}\leq c(\alpha,s,n)\epsilon_0^{\frac12}.$$
The estimates on $B_\nu^{-1}$ and $M_\omega^{-1}$ are similar, we omit it for simplicity.
\end{proof}
It is easy to show   ${M}_\omega^{-1}(\theta)={\overline{M}}^T_\omega(\theta)$ from (\ref{inverseM}) when $\theta\in\T^n$.
Moreover, we can improve Lemma \ref{App02} to the following
\begin{Lemma}\label{App03}
If $\epsilon_0 \ll 1,$ then\\
i) $[B_\nu-M_\omega]_{s,\alpha+}^{\T^n, D_{\varepsilon}},\ [B_\nu^{-1}-M_\omega^{-1} ]_{s,\alpha+}^{\T^n, D_{\varepsilon}}\leq c(\alpha,s,n)\epsilon_\nu^{\frac12};$ \\
ii) $[M_\omega-Id]_{s,\alpha+}^{\T^n, D_{\varepsilon}},\ [M_\omega^{-1} -Id]_{s,\alpha+}^{\T^n, D_{\varepsilon}}\leq c(\alpha,s,n)\epsilon_0^{\frac12}$.
\end{Lemma}
\begin{Remark}
Note that $M_\omega -Id\in \mathcal{M}_{s,\alpha}^+$ and $\alpha>0$, it follows from Lemma \ref{daishu01} that $M_\omega -Id\in \mathfrak{L}(\ell_{1}^2, \ell_{1}^2)$ and satisfies
$$\|(M_\omega -Id)\xi\|_{1}\leq c(\alpha,s)[M_\omega -Id]^{\T^n, D_{\varepsilon}}_{s,\alpha+}\|\xi\|_{1}\leq c(\alpha,s,n)\epsilon_0^{\frac12}\|\xi\|_{1}, $$
for $(\theta,\omega)\in \T^n \times  D_{\varepsilon}$. Thus for $(\theta,\omega)\in \T^n \times  D_{\varepsilon}$
\begin{eqnarray}\label{Momega}
\|M_\omega \xi\|_{1}\leq c(\alpha,s,n)\|\xi\|_{1}.
\end{eqnarray}
Similarly,  for $(\theta,\omega)\in \T^n \times  D_{\varepsilon}$,
\begin{eqnarray}\label{Momegainverse}
\|M_\omega^{-1} \xi\|_{1}\leq c(\alpha,s,n)\|\xi\|_{1}.
\end{eqnarray}
Denote  $\Phi_\omega^\infty(\xi,\eta)=(\overline{M}_\omega\xi,M_\omega\eta)$. From (\ref{Momega}) and (\ref{Momegainverse}) $\Phi_\omega^\infty $ and its inverse are bounded operators from $Y_{1}$ into $Y_{1}$.
\end{Remark}

In the following we denote
$U_{\nu} : = B_{\nu}^{-1}N_0 B_{\nu} $, $U_{\infty} : = M_{\omega}^{-1}N_0 M_{\omega} $,
$V_{\nu} : = B_\nu^{-1}P^{(\nu)}B_\nu $ and $V_{\infty} : = \varepsilon M_{\omega}^{-1}P(\theta)M_{\omega} $ for simplicity.
We will prove that
 \begin{Lemma}\label{convergence02}
  For any $y\in\R^n$,\ $(\xi,\eta)\in Y_{s'}$ with $1\leq s'\leq \max\{s,1\}$,
\begin{equation}\label{phies}
\lim_{\nu\rightarrow\infty}H^{(\nu)}\circ\Phi^\nu(\xi,\eta)=\mathcal{H}\circ\Phi_\omega^\infty(\xi,\eta),
 \end{equation}
 uniformly for $\omega\in D_\varepsilon$ and $\theta\in\T^n$.
 \end{Lemma}
\begin{proof}
 From
the definition, $\Phi^\nu(\xi,\eta)=(\overline{B_\nu}\xi,B_\nu\eta)$ and $(\overline{B_\nu})^T=B_\nu^{-1}$.
Thus,
\begin{eqnarray*}
H^{(\nu)}\circ\Phi^\nu&=&\la \omega,y\ra+\la \overline{B_\nu}\xi,(N_0+P^{(\nu)})B_\nu\eta \ra\\
&=&\la \omega,y\ra+\la \xi,B_\nu^{-1}(N_0+P^{(\nu)})B_\nu\eta \ra\\
&=&\la \omega,y\ra+\la \xi,U_\nu\eta \ra+\la \xi,V_\nu\eta \ra.
\end{eqnarray*}
On the other hand,  by a straightforward computation, we have
\begin{eqnarray*}
\mathcal{H}\circ\Phi_\omega^\infty&=&\la \omega,y\ra+\la \xi,U_\infty\eta \ra+\la \xi,V_\infty\eta \ra.
\end{eqnarray*}
Then (\ref{phies}) is proved by the following two lemmas.
\end{proof}

 \begin{Lemma}\label{convergence03}
  For   $(\xi,\eta)\in Y_{\bar{s}}$ with $\bar{s}\geq 1$,
$$\lim_{\nu\rightarrow\infty}\la \xi,(V_\nu-V_\infty)\eta \ra=0,$$ uniformly for $\omega\in D_\varepsilon,\ \theta\in\T^n$.
\end{Lemma}
\begin{proof}Consider
\begin{eqnarray*}
V_\nu-V_\infty&=&B_\nu^{{-1}}P^{(\nu)}B_\nu-\varepsilon M_\omega^{-1} P {M_\omega}\\
&=&B_\nu^{{-1}}(P^{(\nu)}-\varepsilon P)B_\nu+\varepsilon (B_\nu^{{-1}}-M_\omega^{-1}) P B_\nu+\varepsilon B_\nu^{{-1}}P(B_\nu-{M_\omega})\\
&:=&I_1+I_2+I_3
\end{eqnarray*}
We first estimate $I_1$. Note for $\theta\in\T^n,$ from Lemma \ref{l4.1},
\begin{eqnarray*}\label{pinfty}
[P^{(\nu)}-\varepsilon P]_{s,\alpha}^{D_\varepsilon,\T^n}\leq\sum_{m=\nu}^\infty[P^{(m+1)}-P^{(m)}]_{s,\alpha}^{D_\varepsilon,\T^n}\leq
c(n,\alpha,\beta)\varepsilon\epsilon_{\nu}.
\end{eqnarray*}
On the other hand, by  Lemma \ref{L4.6} we have
\begin{eqnarray}\label{binfty}
 [B_\nu^{{-1}}-Id]_{s,\alpha+}^{D_\varepsilon,\T^n},\
[B_\nu-Id]_{s,\alpha+}^{D_\varepsilon,\T^n}\leq   \epsilon_0^{\frac{1}{2}},
 \end{eqnarray}
then, from Lemma \ref{daishu},
\begin{eqnarray*}\label{pinfty}
[I_1]_{s,\alpha}^{D_\varepsilon,\T^n}&\leq&  [(B_\nu^{{-1}}-Id)(P^{(\nu)}-\varepsilon P)(B_\nu-Id)]_{s,\alpha}^{D_\varepsilon,\T^n}+ [(B_\nu^{{-1}}-Id)(P^{(\nu)}-\varepsilon P)]_{s,\alpha}^{D_\varepsilon,\T^n}\nonumber\\
&+ &[(P^{(\nu)}-\varepsilon P)(B_\nu-Id)]_{s,\alpha}^{D_\varepsilon,\T^n}+ [ P^{(\nu)}-\varepsilon P ]_{s,\alpha}^{D_\varepsilon,\T^n}\nonumber\\
&\leq&  \epsilon_{\nu}.
\end{eqnarray*}
For $I_2$, note that $[B_\nu^{{-1}}-M_\omega^{-1}]_{s,\alpha+}^{D_\varepsilon,\T^n}\leq\epsilon_\nu^{\frac{1}{2}}$ by  Lemma \ref{App03}, thus
\begin{eqnarray*}\label{}
[I_2]_{s,\alpha}^{D_\varepsilon,\T^n}\leq  [\varepsilon (B_\nu^{{-1}}-M_\omega^{-1}) P (B_\nu-Id)]_{s,\alpha}^{D_\varepsilon,\T^n}+[\varepsilon (B_\nu^{{-1}}-M_\omega^{-1}) P  ]_{s,\alpha}^{D_\varepsilon,\T^n}\leq c(s,n,\alpha,\beta)  \varepsilon \epsilon_{\nu}^{\frac{1}{2}}
\end{eqnarray*}
combined with (\ref{binfty}).
Similarly,
$
[I_3]_{s,\alpha}^{D_\varepsilon,\T^n}\leq   c(s,n,\alpha,\beta) \varepsilon \epsilon_{\nu}^{\frac{1}{2}}.
$ Finally, we have
\begin{eqnarray}\label{VnuminusVinfty}
[V_\nu-V_\infty]_{s,\alpha}^{D_\varepsilon,\T^n}\leq   c(s,n,\alpha,\beta)\varepsilon   \epsilon_{\nu}^{\frac{1}{2}},
\end{eqnarray}
and $V_\nu-V_\infty\in\mathfrak{L}(\ell_{\bar{s}}^2,\ell_{-\bar{s}}^2)$.
By Lemma \ref{daishu} and (\ref{VnuminusVinfty}) we obtain
\begin{eqnarray*}
|\la \xi,(V_\nu-V_\infty)\eta\ra|\leq\|\xi\|_{\bar{s}} \|(V_\nu-V_\infty)\eta\|_{-\bar{s}}\leq  c(s,n,\alpha,\beta)\varepsilon \epsilon_{\nu}^{\frac{1}{2}}\|\xi\|_{\bar{s}}\|\eta\|_{\bar{s}}\rightarrow 0
\end{eqnarray*}
as $\nu\rightarrow0$ uniformly for $\omega\in D_\varepsilon,\ \theta\in\T^n$.

\end{proof}

\begin{Lemma}\label{convergence04}
 For  $(\xi,\eta)\in Y_{s'}$ with $1\leq s'\leq \max\{s,1\}$,
$$\lim_{\nu\rightarrow\infty}\la \xi,(U_\nu-U_\infty)\eta \ra=0,$$ uniformly for $\omega\in D_\varepsilon,\ \theta\in\T^n$.
\end{Lemma}
\begin{proof}
Consider
\begin{eqnarray*}
\la \xi, (U_\nu-U_\infty)\eta\ra &=&\la \xi, (B_\nu^{{-1}}N_0 B_\nu-  M_\omega^{-1} N_0 {M_\omega})\eta\ra \\
&=& \la \xi, (B_\nu^{{-1}}-M_\omega^{-1}) N_0 B_\nu \eta\ra +  \la \xi, M_\omega^{-1}N_0({M_\omega} -B_\nu)\eta \ra.
\end{eqnarray*}
\indent We estimate the first term.  For
$ {\eta'}\in\ell_{s'}^2$, $\|N_0\eta'\|_{s'-2}\leq c_1\|\eta'\|_{s'}$ by $\lambda_a\leq c_1w_a$. From Lemma \ref{binfty} and Lemma \ref{daishu01}
we obtain $\|N_0B_{\nu}\eta\|_{s'-2}\leq c(\alpha, s, n)\|\eta\|_{s'}$.
Therefore,
\begin{eqnarray*}\label{}
&&|\la \xi,(B_\nu^{{-1}}-M_\omega^{-1}) N_0 B_\nu\eta \ra|\\
&=&|\la (\overline{B}_{\nu}-\overline{M}_\omega) \xi,N_0 B_\nu\eta \ra|\\
{\rm \underline{from\ Cauchy\  and \ 1\leq s'\leq \max\{s,1\} }} &\leq &  \|(\overline{B}_\nu-\overline{M}_\omega) \xi\|_{s'}\|N_0 B_\nu\eta\|_{s'-2}\\
{\rm \underline{Lemma\  \ref{App03},\  Lemma\  \ref{daishu01}}} &\leq& c(\alpha, s,n)\epsilon_\nu^{\frac{1}{2}}\|\xi\|_{s'}\| \eta\|_{s'}.
\end{eqnarray*}
Similarly, for $(\xi,\eta)\in  Y_{s'}$, we have
$|\la \xi,M_\omega^{-1}N_0({M_\omega} -B_\nu)\eta \ra| \leq  c(\alpha, s,n) \epsilon_\nu^{\frac{1}{2}}\|\xi\|_{s'}\| \eta\|_{s'}$. The conclusion is clear now. \end{proof}
\indent Combining with Lemma \ref{convergence01} and Lemma \ref{convergence02}, we finish the proof of  Lemma \ref{convergence}.\qed \\
\indent From Lemma \ref{convergence} and the concrete form of $\Phi^{\infty}$, we obtain that $$(\omega t, \star, \overline{M}_{\omega}(\omega t)e^{{-\rm i}\overline{N}_{\infty}t}\xi_0, M_{\omega}(\omega t)e^{{\rm i} N_{\infty} t}\eta_0)$$ are the solutions of the Hamiltonian system (\ref{autohs}). Thus, $(\overline{M}_{\omega}(\omega t)e^{{-\rm i}\overline{N}_{\infty}t}\xi_0, M_{\omega}(\omega t)e^{{\rm i} N_{\infty} t}\eta_0)$ are clearly the solutions of the Hamiltonian system (\ref{hameq00}). We complete the proofs of Theorem \ref{MainTheorem}. \qed
\section{Appendix}\label{appendix}
\subsection{Proof of Lemma \ref{psismooth}.}
\begin{proof}
Recall that $$ \Psi_\omega\left(\theta\right)\left(\sum_{a\in\mathcal{E}}\xi_a\Phi_a\left(x\right)\right)=\sum_{a\in\mathcal{E}}\left(M^T_\omega\left(\theta\right)\xi\right)_a\Phi_a\left(x\right).$$ From Theorem \ref{MainTheorem}, it is easy to show by definition $M^T_\omega-Id \in{\mathcal{C}^{\mu}\left(\T^n, \mathfrak{L}\left(\ell^2_{s'}, \ell^2_{s'+2\alpha}\right)\right) }$ and
\begin{eqnarray*}
\|M^T_\omega -Id\|_{\mathcal{C}^{\mu}\left(\T^n, \mathfrak{L}\left(\ell^2_{s'}, \ell^2_{s'+2\alpha}\right)\right) }&\leq&   C\left(n,\beta, \mu, d, s\right) \epsilon^{\frac{3}{2\beta}\left(\frac{2}{9}\beta-\mu\right)}.
\end{eqnarray*}
(a)\quad From definition,
\begin{eqnarray*}
&&\|\Psi_\omega(\cdot)-id\|_{ \mathfrak{L}\left(\mathcal{H}^{s'},\mathcal{H}^{s'+2\alpha}\right)}=\sup_{\|u\|_{\mathcal{H}^{s'}}=1}\|\Psi_\omega(\cdot)u-u\|_{\mathcal{H}^{s'+2\alpha}}\\
&=&\sup_{\|\xi\|_{\ell^2_{s'}}=1}
\|M^T_\omega\xi-\xi\|_{\ell^2_{s'+2\alpha}} = \|M^T_\omega-Id\|_{\mathfrak{L}\left(\ell^2_{s'},\ell^2_{s'+2\alpha}\right)}\leq   C\left(n,\beta, \mu, d, s\right) \epsilon^{\frac{3}{2\beta}\left(\frac{2}{9}\beta-\mu\right)}.
\end{eqnarray*}
(b)\quad For $b=\mu-[\mu]\in\left(0,1\right)$, $z_1,z_2\in \R^{n}$ with $0<|z_1-z_2|<2\pi$,
\begin{eqnarray*}
&&\frac{\|\Psi_\omega\left(z_1\right)-\Psi_\omega\left(z_2\right)\|_{\mathfrak{L}\left(\mathcal{H}^{s'},\mathcal{H}^{s'+2\alpha}\right)}}{{|z_1-z_2|^{b}}}\\
&=&\frac{1}{{|z_1-z_2|^{b}}}\sup_{\|u\|_{\mathcal{H}^{s'}}=1}\|\Psi_\omega \left(z_1\right)u-\Psi_\omega\left(z_2\right)u\|_{ \mathcal{H}^{s'+2\alpha} }\\
&=&\frac{1}{{|z_1-z_2|^{b}}}\sup_{\|\xi\|_{\ell^2_{s'}}=1}\|M^T_\omega\left(z_1\right)\xi-M^T_\omega\left(z_2\right)\xi\|_{ \ell^2_{s'+2\alpha}}\\
&= &\frac{1}{{|z_1-z_2|^{b}}} \|M^T_\omega\left(z_1\right)-M^T_\omega\left(z_2\right)\|_{\mathfrak{L}\left(\ell^2_{s'},\ell^2_{s'+2\alpha}\right)}
\end{eqnarray*}
which shows that
\begin{eqnarray*}
\|\Psi_\omega-id\|_{\mathcal{C}^{b}\left(\T^n, \mathfrak{L}\left(\mathcal{H}^{s'},\mathcal{H}^{s'+2\alpha}\right)\right)}= \|M^T_\omega-Id\|_{\mathcal{C}^{b}\left(\T^n, \mathfrak{L}\left(\ell^2_{s'}, \ell^2_{s'+2\alpha}\right)\right) }\leq   C\left(n,\beta, \mu, d, s\right) \epsilon^{\frac{3}{2\beta}\left(\frac{2}{9}\beta-\mu\right)}.
\end{eqnarray*}
(c)\quad Denote $\la\mathcal{A}\left(z\right),h\ra u:=\sum_{a\in\mathcal{E}}\left(\la A\left(z\right),h\ra\xi\right)_a\Phi_a\left(x\right)$ for $h\in\R^n$ where we use the notation $A:=\left(M^T_\omega-Id\right)'_z$ and $\xi\in \ell_{s'}^2$ and $0\leq s'\leq s$. Note $M_{\omega}^T-Id\in C^{\mu}(\T^n, \mathfrak{L}(\ell_{s'}^2, \ell_{s'+2\alpha}^2))
$,  it follows that  for any $z\in \T^n$, $\mathcal{A}(z)\in \mathfrak{L}(\R^n, \mathfrak{L}(\mathcal{H}^{s'}, \mathcal{H}^{s'+2\alpha}))$.
This is because
\begin{eqnarray}
\|\mathcal{A}\|_{\mathfrak{L}\left(\R^n, \mathfrak{L}\left(\mathcal{H}^{s'},\mathcal{H}^{s'+2\alpha}\right)\right)}&=&\sup_{\|u\|_{\mathcal{H}^{s'}}=1,\|h\|=1} \|\la\mathcal{A}\left(z\right),h\ra u\|_{ \mathcal{H}^{s'+2\alpha} }\nonumber\\
&=&   \sup_{\|\xi\|_{\ell^2_{s'}}=1,\|h\|=1} \|\la A\left(z\right),h\ra\xi\|_{  \ell^2_{s'+2\alpha} }\nonumber\\
&= &   \| A(z)\|_{\mathfrak{L}\left(\R^n, \mathfrak{L}\left(\ell^2_{s'},\ell^2_{s'+2\alpha} \right)\right)}\label{dengnorm}\\
&\leq& \|M^T_\omega-Id\|_{\mathcal{C}^{b}\left(\T^n, \mathfrak{L}\left(\ell^2_{s'}, \ell^2_{s'+2\alpha}\right)\right) }.\nonumber
\end{eqnarray}
Given $z_0\in\T^n$,
\begin{eqnarray*}
&&\|\Psi\left(z\right)-\Psi\left(z_0\right)-\la\mathcal{A}\left(z_0\right),z-z_0\ra\|_{ \mathfrak{L}\left(\mathcal{H}^{s'},\mathcal{H}^{s'+2\alpha}\right)}\\
&=&\sup_{\|u\|_{\mathcal{H}^{s'}}=1} \|\left(\Psi\left(z\right)-\Psi\left(z_0\right)-\la\mathcal{A}\left(z_0\right),z-z_0\ra\right) u\|_{ \mathcal{H}^{s'+2\alpha} }\\
&=&\sup_{\|\xi\|_{\ell^2_{s'}}=1} \|\left(M^T_\omega\left(z\right)-M^T_\omega\left(z_0\right)-\la A\left(z_0\right),z-z_0\ra\right)\xi\|_{  \ell^2_{s'+2\alpha} }\\
&=& \|M^T_\omega\left(z\right)-M^T_\omega\left(z_0\right)-\la A\left(z_0\right),z-z_0\ra\|_{ \mathfrak{L}\left(\ell^2_{s'},\ell^2_{s'+2\alpha} \right)}
\end{eqnarray*}
Note $M^T_\omega-Id \in{\mathcal{C}^{\mu}\left(\T^n, \mathfrak{L}\left(\ell^2_{s'}, \ell^2_{s'+2\alpha}\right)\right) }$, then
$$\|M^T_\omega\left(z\right)-M^T_\omega\left(z_0\right)-\la A\left(z_0\right),z-z_0\ra\|_{ \mathfrak{L}\left(\ell^2_{s'},\ell^2_{s'+2\alpha} \right)}=o\left(|z-z_0|\right),\ z\rightarrow z_0.$$
Therefore, $$\|\Psi\left(z\right)-\Psi\left(z_0\right)-\la\mathcal{A}\left(z_0\right),z-z_0\ra\|_{ \mathfrak{L}\left(\mathcal{H}^{s'},\mathcal{H}^{s'+2\alpha}\right)}=o\left(|z-z_0|\right),\ z\rightarrow z_0,$$
which shows that $\Psi\left(z\right)$ is Fr\'echet differentiable at $z_0.$ Moreover, following (b), we have
\begin{eqnarray*}
&&\sup_{\substack{\|u\|_{\mathcal{H}^{s'}}=1,\|h\|=1\\z_1,z_2\in \R^n, 0<|z_1-z_2|<2\pi}} \frac{1}{{|z_1-z_2|^{b}}} \|\la \mathcal{A}\left(z_1\right)-\mathcal{A}\left(z_2\right),h\ra u\|_{ \mathcal{H}^{s'+2\alpha} }\\
&=&\sup_{\substack{\|\xi\|_{\ell^2_{s'}}=1,\|h\|=1\\z_1,z_2\in \R^n, 0<|z_1-z_2|<2\pi}} \frac{1}{{|z_1-z_2|^{b}}}\|\la  {A}\left(z_1\right)- {A}\left(z_2\right),h\ra \xi\|_{ \ell^2_{s'+2\alpha} }\\
&= &\sup_{ z_1,z_2\in \R^n, 0<|z_1-z_2|<2\pi} \frac{1}{{|z_1-z_2|^{b}}} \| {A}\left(z_1\right)- {A}\left(z_2\right)\|_{\mathfrak{L}\left(\R^n,\mathfrak{L}\left(\ell^2_{s'},\ell^2_{s'+2\alpha}\right)\right)}.
\end{eqnarray*}
Combining with (\ref{dengnorm}), we have
\begin{eqnarray*}
\|\Psi_\omega(\cdot )-id\|_{\mathcal{C}^{1+b}\left(\T^n, \mathfrak{L}\left(\mathcal{H}^{s'},\mathcal{H}^{s'+2\alpha}\right)\right)}
= \|M^T_\omega(\cdot )-Id\|_{\mathcal{C}^{1+b}\left(\T^n, \mathfrak{L}\left(\ell^2_{s'}, \ell^2_{s'+2\alpha}\right)\right) }\leq   C\left(n,\beta, \mu, d, s\right) \epsilon^{\frac{3}{2\beta}\left(\frac{2}{9}\beta-\mu\right)}
\end{eqnarray*}
with $1+b\leq\mu.$  Inductively, we can show that
 \begin{eqnarray*}
\|\Psi_\omega(\cdot)-id\|_{\mathcal{C}^{k+b}\left(\T^n, \mathfrak{L}\left(\mathcal{H}^{s'},\mathcal{H}^{s'+2\alpha}\right)\right)}
= \|M^T_\omega(\cdot )-Id\|_{\mathcal{C}^{k+b}\left(\T^n, \mathfrak{L}\left(\ell^2_{s'}, \ell^2_{s'+2\alpha}\right)\right) }\leq   C\left(n,\beta, \mu, d, s\right) \epsilon^{\frac{3}{2\beta}\left(\frac{2}{9}\beta-\mu\right)}
\end{eqnarray*}
with $k+b\leq\mu.$ Thus we finish the proof of Lemma \ref{psismooth}.
\end{proof}

\subsection{Proof of Lemma \ref{smoothinginverse}}
\begin{proof}
Following Salamon\cite{Sal04}, it is enough to consider the case $\mu=\ell$. Moreover, once the result has been established for $0<\ell<1$ it follows for $\ell>1$  by
Cauchy's estimate. Therefore we assume $0<\iota=\mu=\ell<1$.\\
\indent Define $g_\nu=f_\nu-f_{\nu-1}$. Then $f=\sum g_\nu$ satisfies the estimate
\begin{eqnarray*}
|f|_{X}\leq c\sum_{\nu=1}^\infty\sigma_\nu^{\iota}=c\sum_{\nu=1}^\infty\sigma^{\iota(\frac{3}{2})^\nu}\leq \frac{2c}{\iota}\sigma^{\iota},
\end{eqnarray*}
where we use the fact that $\sum_{\nu=1}^\infty\sigma^{\iota(\frac{3}{2})^\nu}\leq \sigma^{\iota}\sum_{\nu\geq0}(\frac{1}{2^\nu})^{\iota}\leq \frac{2 }{\iota}\sigma^{\iota}$ for $0<\sigma\leq1/4.$\\
\indent For $x,y\in\R^n$ with $\sigma<|x-y|\leq1$ this implies
$
|f(x)-f(y)|_{X}\leq \frac{4c}{\iota}\sigma^{\iota}\leq \frac{4c}{\iota}|x-y|^{\iota}.
$
In the case $0<|x-y|\leq\sigma$ there is an integer $N\geq0$ such that
$\sigma_{N+1}<|x-y|\leq\sigma_N$.
Following Cauchy's estimate, $|\partial_xg_\nu(u)|_{X}\leq c\sigma_\nu^{\iota-1}$ for every $u\in\R^n$, we have
$
|g_\nu(x)-g_\nu(y)|_{X}\leq c\sigma_\nu^{\iota-1}|x-y|.
$
We shall use this estimate for $\nu=1,2,\cdots,N$. For $\nu\geq N+1$ we use the trivial estimate
$
|g_\nu(x)-g_\nu(y)|_{X}\leq 2c\sigma_\nu^{\iota}.
$
Taking into account the inequalities we obtain that
\begin{eqnarray*}
|f(x)-f(y)|_{X}&\leq& \sum_{1\leq\nu\leq N}|g_\nu(x)-g_\nu(y)|_{X}+\sum_{\nu> N}|g_\nu(x)-g_\nu(y)|_{X}\\
&\leq& c|x-y|\sum_{1\leq\nu\leq N}\sigma_\nu^{\iota-1}+2c\sum_{\nu> N}\sigma_\nu^{\iota}\\
&\leq&  c|x-y|\sigma_N^{\iota-1}\sum_{0\leq\nu\leq N-1}\left(\frac{1}{2^\nu}\right)^{1-\iota}+2c\sigma_{N+1}^{\iota}\sum_{\nu\geq0}\left(\frac{1}{2^\nu}\right)^{\iota}\\
&\leq& \frac{4c}{1-\iota}|x-y|\sigma_N^{\iota-1}+\frac{4c}{\iota}\sigma_{N+1}^{\iota}\\
&=& \frac{4c}{\iota(1-\iota)}|x-y|^{\iota}.
\end{eqnarray*}
We finish the proof.
\end{proof}

\end{document}